\documentclass[12pt,a4paper]{article}
\usepackage{}
\usepackage{amsfonts}
\usepackage{amssymb}
\usepackage{mathrsfs}
\usepackage{amsmath}
\usepackage{amsthm}
\usepackage{enumerate}
\usepackage{cite}
\usepackage[all]{xy}
\allowdisplaybreaks
\newtheorem{defn}{Definition}[section]
\newtheorem{thm}[defn]{Theorem}
\newtheorem{lem}[defn]{Lemma}
\newtheorem{prop}[defn]{Proposition}
\newtheorem{cor}[defn]{Corollary}
\newtheorem{eg}[defn]{Example}
\newtheorem{re}[defn]{Remark}

\newcommand\relphantom[1]{\mathrel{\phantom{#1}}}

\newcommand{\bdefn}{\begin{defn}}
\newcommand{\edefn}{\end{defn}}
\newcommand{\bthm}{\begin{thm}}
\newcommand{\ethm}{\end{thm}}
\newcommand{\blem}{\begin{lem}}
\newcommand{\elem}{\end{lem}}
\newcommand{\bprop}{\begin{prop}}
\newcommand{\eprop}{\end{prop}}
\newcommand{\bcor}{\begin{cor}}
\newcommand{\ecor}{\end{cor}}
\newcommand{\beg}{\begin{eg}}
\newcommand{\eeg}{\end{eg}}
\newcommand{\bre}{\begin{re}}
\newcommand{\ere}{\end{re}}
\newcommand{\bpf}{\begin{proof}}
\newcommand{\epf}{\end{proof}}

\newcommand{\hg}{\rm hg}
\newcommand{\Hom}{\rm Hom}

\newcommand{\Chom}{\rm Chom}
\newcommand{\Z}{{\rm Z}}
\newcommand{\C}{{\rm C}}
\newcommand{\QC}{{\rm QC}}
\newcommand{\GDer}{{\rm GDer}}
\newcommand{\ZDer}{{\rm ZDer}}
\newcommand{\Cend}{{\rm Cend}}
\newcommand{\QDer}{{\rm QDer}}
\newcommand{\Der}{{\rm Der}}

\newcommand{\id}{{\rm id}}

\newcommand{\End}{{\rm End}}

\newcommand{\R}{\mathcal{R}}

\newcommand{\g}{\mathfrak{g}}

\newcommand{\benu}{\begin{enumerate}}
\newcommand{\eenu}{\end{enumerate}}
\newcommand{\bc}{\begin{center}}
\newcommand{\ec}{\end{center}}
\newcommand{\bea}{\begin{eqnarray}}
\newcommand{\eea}{\end{eqnarray}}
\newcommand{\Bea}{\begin{eqnarray*}}
\newcommand{\Eea}{\end{eqnarray*}}
\newcommand{\beq}{\begin{equation}}
\newcommand{\eeq}{\end{equation}}
\newcommand{\Beq}{\begin{equation*}}
\newcommand{\Eeq}{\end{equation*}}
\newcommand{\bspl}{\begin{split}}
\newcommand{\espl}{\end{split}}

\numberwithin{equation}{section}

\begin{document}
\title{{\bf Deformations and generalized derivations of Lie conformal superalgebras}}
\author{ Jun Zhao$^{1}$, Liangyun Chen$^{1*}$, Lamei Yuan$^{2}$
 \date{{\small $^{1}$School of Mathematics and Statistics, Northeast Normal
 University,\\
Changchun 130024, China\\
 $^{2}$Academy of Fundamental and Interdisciplinary Sciences,\\ Harbin Institute of Technology, Harbin 150080, China}}}

\maketitle
\date{}

\begin{abstract}
The purpose of this paper is to extend the cohomology and conformal derivation theories of the classical Lie conformal algebras to Lie conformal superalgebras. Firstly, we construct the semidirect product of a Lie conformal superalgebra and its conformal module, and study derivations of this semidirect product.
Secondly, we develop cohomology theory of Lie conformal superalgebras and discuss some applications to the study of deformations of Lie conformal superalgebras. Finally, we introduce generalized derivations of Lie conformal superalgebras and study their properties.
\bigskip

\end{abstract}
\renewcommand{\thefootnote}{\fnsymbol{footnote}}
\footnote[0]{ Corresponding author(L. Chen): chenly640@nenu.edu.cn.}
\footnote[0]{Supported by NNSF of China (Nos. 11171055, 11471090 and 11301109).}

\section{Introduction}

Lie conformal superalgebras encode the singular part of the operator product expansion of
chiral fields in two-dimensional quantum field theory. On the other hand, they are closely connected
to the notion of a formal distribution Lie superalgebra $(\g,F)$, that is a Lie superalgebra $\g$
spanned by the coefficients of a family $F$ of mutually local formal distributions. Conversely, to a Lie
conformal superalgebra $\R$ one can associate a formal distribution Lie superalgebra $(Lie \R,\R)$
which establishes an equivalence between the category of Lie conformal superalgebras and the
category of equivalence classes of formal distribution Lie superalgebras obtained as quotients of
Lie $\R$ by irregular ideals. See \cite{BKLR,Z,FKR} in details.

The notion of a Lie conformal superalgebra was introduced in \cite{FKR}, which gave derivations and representations of Lie conformal superalgebras. The cohomology theory of Lie conformal algebras was developed in \cite{BKV}. Lately, the generalized derivation theory of Lie conformal algebras were studied in \cite{FHS}, extending the generalized derivation theory of Lie algebras given in \cite{L}. \cite{GT} study conformal derivations of the semidirect product. \cite{LCM} study Hom-Nijienhuis operators and $T$*-extensions of hom-Lie superalgebras. In the present paper, we aim to do same in \cite{GT} for Lie conformal superalgebras, extend the cohomology theory and study the generalized derivations of Lie conformal algebras to the super case.  Furthermore, by a $2$-cocycle Nijienhuis operator
generate a deformation.

The paper is organized as follows. In Section 2, we recall notions of derivations and modules of a Lie conformal superalgebra. Moreover, we construct the semidirect product of a Lie conformal superalgebra and its conformal module, and study derivations of this semidirect product.

In Section 3, we define cohomology and Nijienhuis operators of Lie conformal superalgebras,
and show that the deformation generated by a $2$-cocycle Nijienhuis operator is trivial.

In Section 4, we introduce different kinds of generalized derivations of Lie conformal superalgebras, and study their properties and connections, extending some results obtained in \cite{L}.

Throughout this paper, all vector
spaces, linear maps, and tensor products are over the complex field $\mathbb{C}$. In addition to the standard
notations $\mathbb{Z}$ and $\mathbb{R}$, we use $\mathbb{Z}_+$ to denote the set of nonnegative integers. $\R$ is a $\mathbb{Z}_{2}$-graded $\mathbb{C}[\partial]$-module, $x\in\R$ means $x$ is a homogeneous element, $|a|\in\mathbb{Z}_{2}$ is the degree of $a$.

\section{Conformal derivations of semidirect products of Lie conformal superalgebras and their conformal modules}

First we present the definition of Lie conformal superalgebras given in \cite[Definition 2.1]{BKLR}.
\bdefn \rm\label{def2-1}
A Lie conformal superalgebra $\R$ is a left $\mathbb{Z}_{2}$-graded $\mathbb{C}[\partial]$-module, and for any $n\in\mathbb{Z}_{\geq0}$ there is a
family of $\mathbb{C}$-bilinear  $n$-products from $\R\otimes\R$ to $\R$ satisfying the following conditions:
\begin{enumerate}
\item[$(\rm C0)$]  For any $a,b\in\R$, there is an $N$ such that $a_{(n)}b=0$ for $n\gg N$,
\item[$(\rm C1)$]  For any $a,b\in\R$ and $n\in\mathbb{Z}_{\geq0}$, $(\partial a)_{(n)}b=-n a_{(n)}b$,
\item[$(\rm C2)$]  For any $a,b\in\R$ and $n\in\mathbb{Z}_{\geq0}$, $$a_{(n)}b=-(-1)^{|a||b|}\sum_{j=0}^{\infty}(-1)^{j+n}\frac{1}{j!}\partial^j(b_{(n+j)}a),$$
\item[$(\rm C3)$]  For any $a,b,c\in\R$ and $m,n\in\mathbb{Z}_{\geq0}$,
$$a_{(m)}(b_{(n)}c)=\sum_{j=0}^{m}(_{j}^{m})(a_{(j)}b)_{(m+n-j)}c+(-1)^{|a||b|}b_{(n)}(a_{(m)}c).$$
\end{enumerate}

Note that if we define $\lambda$-bracket $[-_{\lambda}-]$:
\begin{eqnarray}
[a_\lambda b]=\sum_{n=0}^{\infty}\frac{\lambda^n}{n!}a_{(n)}b,\, a,b\in\R.
\end{eqnarray}
That is, $\R$ is a Lie conformal superalgebra if and only if $[-_{\lambda}-]$ satisfies the following axioms:
\begin{eqnarray*}
&&(\rm C1)_{\lambda}\ \ \ \ {\rm  Conformal\ sesquilinearity}:\ \, [(\partial a)_\lambda b]=-\lambda[a_\lambda b];\\
&&(\rm C2)_{\lambda}\ \ \ \ {\rm Skew-symmetry}:     \ \ \ \ \ \ \ \ \ \ \ \, [a_\lambda b]=-(-1)^{|a||b|}[b_{-\partial-\lambda}a];\\
&&(\rm C3)_{\lambda}\ \ \ \ {\rm Jacobi\ identity}:  \ \ \ \ \ \ \ \ \ \ \ \ \ \ \  \, [a_\lambda[b_\mu c]]=[[a_\lambda b]_{\lambda+\mu}c]+(-1)^{|a||b|}[b_\mu[a_\lambda c]].
\end{eqnarray*}
\edefn

\beg\rm Let $\R=\mathbb{C}[\partial]L\oplus\mathbb{C}[\partial]E$ be a free $\mathbb{Z}_{2}$-graded $\mathbb{C}[\partial]$-module. Define
\begin{eqnarray*}
[L_\lambda L]=(\partial+2\lambda)L,\, [L_\lambda E]=(\partial+\frac{3}{2}\lambda)E,
\, [E_\lambda L]=(\frac{1}{2}\partial+\frac{3}{2}\lambda)E,\, [E_\lambda E]=0,
\end{eqnarray*}
where ${\R}_{\bar{0}}=\mathbb{C}[\partial]L$ and ${\R}_{\bar{1}}=\mathbb{C}[\partial]E$. Then $\R$ is a Lie conformal superalgebra.
\eeg

\beg\rm\label{exp2-1}(Neveu-Schwarz Lie conformal superalgebra) Let $\rm{NS}=\mathbb{C}[\partial]L\oplus\mathbb{C}[\partial]G$ be a free $\mathbb{Z}_{2}$-graded $\mathbb{C}[\partial]$-module. Define
\begin{eqnarray*}
[L_\lambda L]=(\partial+2\lambda)L,\, [L_\lambda G]=(\partial+\frac{3}{2}\lambda)G,
\, [G_\lambda L]=(\frac{1}{2}\partial+\frac{3}{2}\lambda)G,\, [G_\lambda G]=L,
\end{eqnarray*}
where $\rm{NS}_{\bar{0}}=\mathbb{C}[\partial]L$ and $\rm{NS}_{\bar{1}}=\mathbb{C}[\partial]G$. Then $\rm{NS}$ is a Lie conformal superalgebra.
We call it Neveu-Schwarz Lie conformal superalgebra.
\eeg

\begin{prop}
Let $\R=\mathbb{C}[\partial]x\oplus\mathbb{C}[\partial]y$ be a Lie conformal superalgebra that is a free $\mathbb{Z}_{2}$-graded $\mathbb{C}[\partial]$-module, where ${\R}_{\bar{0}}=\mathbb{C}[\partial]x$ and ${\R}_{\bar{1}}=\mathbb{C}[\partial]y$. Then $\R$ is isomorphic to one of the following Lie conformal superalgebras:
\begin{enumerate}
\item[$(1)$]  ${\R}_1:[x_\lambda x]=0,\, [x_\lambda y]=0, \, [y_\lambda y]=p(\partial)x, \forall \, p(\partial)\in\mathbb{C}[\partial]$;
\item[$(2)$]  ${\R}_2:[x_\lambda x]=0,\, [y_\lambda y]=0, \, [x_\lambda y]=q(\lambda)y, \forall \, q(\lambda)\in\mathbb{C}[\lambda]$;
\item[$(3)$]  ${\R}_3:[x_\lambda x]=(\partial+2\lambda)x,\, [x_\lambda y]=0, \, [y_\lambda y]=0$;
\item[$(4)$]  ${\R}_4:[x_\lambda x]=(\partial+2\lambda)x,\, [x_\lambda y]=(\partial+\beta\lambda)y, \, [y_\lambda y]=0, \forall \, \beta\in\mathbb{C}$;
\item[$(5)$]  ${\R}_5:[x_\lambda x]=(\partial+2\lambda)x,\, [x_\lambda y]=(\partial+\frac{3}{2}\lambda)y, \, [y_\lambda y]=\alpha x, \forall \, \alpha\in\mathbb{C}$.
\end{enumerate}
\end{prop}
\begin{proof}
Let us set $[x_\lambda x]=a(\partial,\lambda)x$, $[x_\lambda y]=b(\partial,\lambda)y$, and $[y_\lambda y]=d(\partial)x$ (because of $[y_\lambda y]=[y_{-\partial-\lambda}y]$), where $a(\partial,\lambda)=\sum_{i=0}^n a_i(\lambda)\partial^i$, $b(\partial,\lambda)=\sum_{i=0}^m b_i(\lambda)\partial^i$, and $d(\partial,\lambda)=\sum_{i=0}^q d_i\partial^i,$ with $a_n(\lambda)\neq 0, b_m(\lambda)\neq 0$, and $d_q\neq 0.$ Consider the Jacobi identity
$[x_\lambda[x_\mu x]]=[[x_\lambda x]_{\lambda+\mu}x]+[x_\mu[x_\lambda x]],$ this is equivalent to saying that
\begin{eqnarray}
a(\partial+\lambda,\mu)a(\partial,\lambda)x=a(-\lambda-\mu,\lambda)a(\partial,\lambda+\mu)x+a(\partial+\mu,\lambda)a(\partial,\mu)x.\label{prop2-41}
\end{eqnarray}
Equating terms of degree $2n-1$ in $\partial$ in both sides, we get, if $n>1: n(\lambda-\mu)a_n(\lambda)a_n(\mu)=0$, which shows that $a_n(\lambda)=0$.
So $a(\partial,\lambda)=a_0(\lambda)+a_1(\lambda)\partial$. If we put $\lambda=\mu$ in Eq.(\ref{prop2-41}), we get $a(\partial,2\lambda)a(-2\lambda,\lambda)=0$, hence $a_0(\lambda)=2\lambda a_1(\lambda)$, therefore $a(\partial,\lambda)=a_1(\lambda)(\partial+2\lambda).$ Plugging this into $[x_\lambda x]=-[x_{-\partial-\lambda}x]$, we get that $a_1(\lambda)$ is a constant $C$. So $[x_\lambda x]=C(\partial+2\lambda)x$.

Case $1$

If $C=0$, Consider the Jacobi identity
$[x_\lambda[x_\mu y]]=[[x_\lambda x]_{\lambda+\mu}y]+[x_\mu[x_\lambda y]],$ this is equivalent to saying that
\begin{eqnarray*}
b(\partial+\lambda,\mu)b(\partial,\lambda)y=b(\partial+\mu,\lambda)b(\partial,\mu)y.
\end{eqnarray*}
Equating terms of degree $2m-1$ in $\partial$ in both sides, we get, if $m\geq1: m(\lambda-\mu)b_m(\lambda)b_m(\mu)=0$, which shows that $b_m(\lambda)=0$. So $b(\partial,\lambda)=b_0(\lambda)$. So $[x_\lambda y]=b_0(\lambda)y$.

$(1)$ If $b_0(\lambda)=0$, then $[x_\lambda y]=[y_\lambda x]=0$. In this situation, $\R$ is isomorphic to ${\R}_1$.

$(2)$ If $b_0(\lambda)\neq0$, consider the Jacobi identity
$[x_\lambda[y_\mu y]]=[[x_\lambda y]_{\lambda+\mu}y]+[y_\mu[x_\lambda y]],$ this is equivalent to saying that
\begin{eqnarray*}
0=b_0(\lambda)d(\partial)x+b_0(\lambda)d(\partial)x.
\end{eqnarray*}
So $d(\partial)=0$, thus $[y_\lambda y]=0$. Then $\R$ is isomorphic to ${\R}_2$.

Case $2$

If $C\neq0$, we can let $C=1$(otherwise, change $x$ by a complex multiple). Then $[x_\lambda x]=(\partial+2\lambda)x$.
Consider the Jacobi identity
$[x_\lambda[x_\mu y]]=[[x_\lambda x]_{\lambda+\mu}y]+[x_\mu[x_\lambda y]],$ this is equivalent to saying that
\begin{eqnarray}
b(\partial+\lambda,\mu)b(\partial,\lambda)y=(\lambda-\mu)b(\partial,\lambda+\mu)y+b(\partial+\mu,\lambda)b(\partial,\mu)y.\label{prop2-42}
\end{eqnarray}
Equating terms of degree $2m-1$ in $\partial$ in both sides, we get, if $m>1: m(\lambda-\mu)b_m(\lambda)b_m(\mu)=0$, which shows that $b_m(\lambda)=0$.
So $b(\partial,\lambda)=b_0(\lambda)+b_1(\lambda)\partial$. Substituting into Eq.(\ref{prop2-42}), we get
\begin{eqnarray}b_1(\lambda)b_1(\mu)=b_1(\lambda+\mu) \label{prop2-43}
\end{eqnarray} and
\begin{eqnarray}\lambda b_0(\lambda)b_1(\mu)=(\lambda-\mu)b_0(\lambda+\mu)+\mu b_1(\lambda)b_0(\mu).\label{prop2-44}\end{eqnarray}
By Eq.(\ref{prop2-43}), we have $b_1(\lambda)=0$ or $1$.

$(1)$ If $b_1(\lambda)=0$, then $(\lambda-\mu)b_0(\lambda+\mu)=0$, so $b_0(\lambda)=0$, that is $[x_\lambda y]=0$. Consider the Jacobi identity
$[x_\lambda[y_\mu y]]=[[x_\lambda y]_{\lambda+\mu}y]+[y_\mu[x_\lambda y]],$ this is equivalent to saying that
$d(\partial+\lambda)(\partial+2\lambda)=0$, i.e. $[y_\lambda y]=0$. Then $\R$ is isomorphic to ${\R}_3$.

$(2)$ If $b_1(\lambda)=1$, Substituting into Eq.(\ref{prop2-44}), we get $\lambda b_0(\lambda)-\mu b_0(\mu)=(\lambda-\mu)b_0(\lambda+\mu)$.
So $b_0(\lambda)=\beta\lambda$, and $[x_\lambda y]=(\partial+\beta\lambda)y$. Consider the Jacobi identity
$[x_\lambda[y_\mu y]]=[[x_\lambda y]_{\lambda+\mu}y]+[y_\mu[x_\lambda y]],$ this is equivalent to saying that
\begin{eqnarray}
(\partial+2\lambda)d(\partial+\lambda)x=(\partial-\lambda+2\beta\lambda)d(\partial)x.\label{prop2-45}
\end{eqnarray}
Equating terms of the most degree in $\lambda$ in both sides, we get $d(\partial)=C_1$.\\
$(i)$ If $C_1=0$, then $[y_\lambda y]=0$. Then $\R$ is isomorphic to ${\R}_4$.\\
$(ii)$ If $C_1\neq0$, Substituting into Eq.(\ref{prop2-45}), we get $\beta=\frac{3}{2}$. Then $\R$ is isomorphic to ${\R}_5$.
\end{proof}

\bdefn\rm
A linear map $\rho:\R\rightarrow\R'$ is a homomorphism of Lie conformal superalgebras if $\rho$ satisfies
$\rho\partial=\partial\rho$ and $\rho([a_{\lambda}b])=[\rho(a)_{\lambda}\rho(b)], \forall a,b \in\R.$
\edefn

\beg\rm\label{exp2-3}\cite{FKR}
Let $\g=\g_{\bar{0}}\oplus\g_{\bar{1}}$ be a complex Lie superalgebra with Lie bracket $[-,-]$.
Let $\rm (Cur\g)_{\theta}:=\mathbb{C}[\partial]\otimes\g_{\theta}$ be the free $\mathbb{C}[\partial]$-module. Then $\rm Cur\g=(Cur\g)_{\bar{0}}\oplus(Cur\g)_{\bar{1}}$ is a Lie conformal superalgebra, called current Lie conformal superalgebra, with
$\lambda$-bracket given by:
\begin{eqnarray}
[(f(\partial)\otimes a)_\lambda(g(\partial)\otimes b)]:=f(-\lambda)g(\partial+\lambda)\otimes[a,b], a,b\in\R.
\end{eqnarray}
\eeg

\bdefn\rm\label{def2-4}\cite{FKR}
Let $M$ and $N$ be $\mathbb{Z}_{2}$-graded $\mathbb{C}[\partial]$-modules. A conformal linear map of degree $\theta$ from $M$ to $N$ is a sequence $f=\{f_{(n)}\}_{n\in\mathbb{Z}_{\geq0}}$ of $f_{(n)}\in\Hom_{\mathbb{C}}(M,N)$ satisfying that
 $$\partial_{N} f_{(n)}-f_{(n)}\partial_{M}=-n f_{(n-1)},\ n\in\mathbb{Z}_{\geq0} \ \ \ {\rm  and} \ \ \
f_{(n)}(M_{\mu})\subseteq N_{\mu+\theta},  \ \mu,\theta\in\mathbb{Z}_{2}.$$
 Set $f_\lambda=\sum_{n=0}^{\infty}\frac{\lambda^n}{n!}f_{(n)}$. Then
 $f=\{f_{(n)}\}_{n\in\mathbb{Z}_{\geq0}}$ is a conformal linear map of degree $\theta$ if and only if
 $$f_{\lambda}\partial_{M}=(\partial_{N}+\lambda) f_{\lambda} \ \ \ {\rm and} \ \ \
f_{\lambda}(M_{\mu})\subseteq N_{\mu+\theta}[\lambda],  \ \mu,\theta\in\mathbb{Z}_{2}.$$
\edefn

\bdefn\rm
An associative conformal superalgebra $\R$ is a left $\mathbb{Z}_{2}$-graded $\mathbb{C}[\partial]$-module endowed with a $\lambda$-product
from $\R\otimes\R$ to $\mathbb{C}[\lambda]\otimes\R$, for any $a,b,c\in\R$,
satisfying the following conditions:
\begin{enumerate}
\item [$(1)$] $(\partial a)_\lambda b=-\lambda a_\lambda b,\ a_\lambda(\partial b)=(\partial+\lambda)(a_\lambda b)$,
\item [$(2)$] $a_\lambda(b_\mu c)=(a_\lambda b)_{\lambda+\mu}c$.
\end{enumerate}
\edefn

Let $\Chom(M,N)_{\theta}$ denote the set of conformal linear maps of degree $\theta$ from $M$ to $N$. Then $\Chom(M,N)=\Chom(M,N)_{\bar{0}}\oplus\Chom(M,N)_{\bar{1}}$ is a $\mathbb{Z}_{2}$-graded $\mathbb{C}[\partial]$-module via:
\begin{eqnarray*}\partial f_{(n)}=-n f_{(n-1)},\ {\rm equivalently},\ \partial f_{\lambda}=-\lambda f_{\lambda}.\end{eqnarray*}
The composition $f_{\lambda}g:L\rightarrow N\otimes\mathbb{C}[\lambda]$ of conformal linear maps $f:M\rightarrow N$ and $g:L\rightarrow M$ is given by
 \begin{eqnarray*}
(f_{\lambda}g)_{\lambda+\mu}=f_{\lambda}g_{\mu},\ \ \ \forall \, f,g\in \Chom(M,N).
 \end{eqnarray*}

If $M$ is a finitely generated $\mathbb{Z}_{2}$-graded $\mathbb{C}[\partial]$-module, then $\Cend(M):=\Chom(M,M)$ is an associative conformal superalgebra with respect to the above composition. Thus, $\Cend(M)$ becomes a Lie conformal superalgebra, denoted as $gc(M)$, with respect to the following $\lambda$-bracket(see \cite[Example1.1]{FKR}):
\begin{eqnarray}\label{def2-5}
[f_{\lambda}g]_{\mu}=f_{\lambda}g_{\mu-\lambda}-(-1)^{|f||g|}g_{\mu-\lambda}f_{\lambda}.
\end{eqnarray}
Hereafter all $\mathbb{Z}_{2}$-graded $\mathbb{C}[\partial]$-modules are supposed to be  finitely generated.

\bdefn\label{def2.6}\rm \cite[Definition 2.2]{BKLR}
A $\mathbb{Z}_{2}$-graded $\mathbb{C}[\partial]$-module $M$ is a conformal module of a Lie conformal superalgebra $\R$ if there is a homomorphism of Lie conformal superalgebras $\rho:\R\rightarrow\Cend(M)$.
\edefn
Furthermore, we call $(\rho,M)$ is a representation of the Lie conformal superalgebra $\R$. If $(\rho,M)$ is a representation of a Lie conformal superalgebra $\R$, it is obvious that we have the following relations;
\begin{eqnarray*}
\rho(a)_{(m)}\rho(b)_{(n)}=\sum_{j=0}^{m}(_{j}^{m})\rho(a_{(j)}b)_{(m+n-j)}+(-1)^{|a||b|}\rho(b)_{(n)}\rho(a)_{(m)},
\rho(\partial(a))_{(n)}=-n\rho(a)_{(n)};
\end{eqnarray*}
equivalently,
\begin{eqnarray*}
\rho(a)_{\lambda}\rho(b)_{\mu}-(-1)^{|a||b|}\rho(b)_{\mu}\rho(a)_{\lambda}=[\rho(a)_{\lambda}\rho(b)]_{\lambda+\mu}
=\rho([a_{\lambda}b])_{\lambda+\mu},
\rho(\partial(a))_{\lambda}=-\lambda\rho(a)_{\lambda}.
\end{eqnarray*}

\beg\rm\label{exp2-7}
Let $\g=\g_{\bar{0}}\oplus\g_{\bar{1}}$ be a finite dimensional complex Lie superalgebra, $\pi:\g\rightarrow \End(M)$ a finite dimensional representation of $\g$. Then the free $\mathbb{Z}_{2}$-graded $\mathbb{C}[\partial]$-module $\mathbb{C}[\partial]\otimes M$ is a conformal module of
$\rm Cur\g$, with module structure $\rho:{\rm Cur\g}\rightarrow \Cend(\mathbb{C}[\partial]\otimes M)$ given by:
\begin{eqnarray*}
\rho(f(\partial)\otimes a)_{\lambda}(g(\partial)\otimes m)=f(-\lambda)g(\partial+\lambda)\otimes\pi(a)(m), \forall\,f(\partial),g(\partial)\in\mathbb{C}[\partial], a\in\g, m\in M.
\end{eqnarray*}
\eeg

\begin{defn}\rm\cite{FKR}
Let $\R$ be a Lie conformal superalgebra. $d\in\Cend(\R)$ is a conformal derivation if for any $a,b\in\R$ it holds that
\begin{eqnarray*}
d_{(m)}(a_{(n)}b)=\sum_{j=0}^{m}(_{j}^{m})(d_{(j)}a)_{(m+n-j)}b+(-1)^{|a||d|}a_{(n)}(d_{(m)}(b));
\end{eqnarray*}
equivalently,
\begin{eqnarray}
d_\lambda([a_\mu b])=[(d_\lambda(a))_{\lambda+\mu}b]+(-1)^{|a||d|}[a_\mu(d_\lambda(b))].\label{defa}
\end{eqnarray}
\end{defn}
For any $r\in\R$, $d^{r}_{\lambda}$ is called an inner conformal derivation of $\R$ if $d^{r}_{\lambda}(r')=[r_{\lambda}r']$, $\forall r'\in\R$.

\beg\rm Let $\R=\mathbb{C}[\partial]L\oplus{\rm Cur\g}$ be a free $\mathbb{Z}_{2}$-graded $\mathbb{C}[\partial]$-module, where ${\rm Cur\g}$ is current Lie conformal superalgebra. Define $d^L:{\rm Cur\g}\rightarrow{\rm Cur\g}$ by ${d}^L_\lambda g=(\partial+\lambda)g$ for every $g\in\g\subset {\rm Cur\g}$. Then $d^L$ is a conformal derivation degree $0$ of ${\rm Cur\g}$. Furthermore, if we define
\begin{eqnarray*}
[L_\lambda L]=(\partial+2\lambda)L,\, [L_\lambda g]={d}^L_\lambda g,
\, [g_\lambda h]=[g,h],
\end{eqnarray*}
where ${\R}_{\bar{0}}=\mathbb{C}[\partial]L\oplus{\rm Cur\g_{\bar{0}}}$ and ${\R}_{\bar{1}}={\rm Cur\g_{\bar{1}}}$. Then $\R$ is a Lie conformal superalgebra.
\eeg

\beg\rm\label{exp2-10}
Every conformal derivation of Neveu-Schwarz Lie conformal superalgebra is an inner conformal derivation.
\begin{proof}
Case $1$

Suppose that $|d|=0$, and $d_\lambda L=p(\partial,\lambda)L$, $d_\lambda G=g(\partial,\lambda)G$, where $p(\partial,\lambda)=\sum_{i=0}^n a_i(\lambda)\partial^i$, $g(\partial,\lambda)=\sum_{i=0}^q d_i(\lambda)\partial^i$. If $d$ is a conformal derivation, it needs to satisfy $d_\lambda[L_\mu L]=[(d_\lambda L)_{\lambda+\mu}L]+[L_\mu(d_\lambda L)]$, this is equivalent to saying that
\begin{eqnarray}
(\partial+\lambda+2\mu)p(\partial,\lambda)L=p(-\lambda-\mu,\lambda)(\partial+2\lambda+2\mu)L+p(\partial+\mu,\lambda)(\partial+2\mu)L.\label{exp2-101}
\end{eqnarray}
Equating terms of degree $n$ in $\partial$ in both sides, we get, if $n>1: (\lambda-n\mu)a_n(\lambda)=0$, which shows that $a_n(\lambda)=0$.
So $p(\partial,\lambda)=a_0(\lambda)+a_1(\lambda)\partial$, substituting into Eq.(\ref{exp2-101}), we get $a_0(\lambda)=2\lambda a_1(\lambda)$, hence $d_\lambda L=a_1(\lambda)(\partial+2\lambda)L$.

Similarly, $d_\lambda[G_\mu G]=[(d_\lambda G)_{\lambda+\mu}G]+[G_\mu(d_\lambda G)]$, it is equivalent to
\begin{eqnarray}
p(\partial,\lambda)L=g(-\lambda-\mu,\lambda)L+g(\partial+\mu,\lambda)L.\label{exp2-102}
\end{eqnarray}
The degree of $\partial$ in $p(\partial,\lambda)$ is at most one, so $g(\partial,\lambda)=d_0(\lambda)+d_1(\lambda)\partial$, substituting into Eq.(\ref{exp2-102}), we get $d_0(\lambda)=\frac{3}{2}\lambda d_1(\lambda)=\frac{3}{2}\lambda a_1(\lambda)$, hence $d_\lambda G=a_1(\lambda)(\partial+\frac{3}{2}\lambda)G$.

So $d_\lambda=a_1(\lambda){\rm ad}L_\lambda$.

Case $2$

Suppose that $|d|=1$, and $d_\lambda L=q(\partial,\lambda)G$, $d_\lambda G=f(\partial,\lambda)L$, where $q(\partial,\lambda)=\sum_{i=0}^m b_i(\lambda)\partial^i$, $f(\partial,\lambda)=\sum_{i=0}^p c_i(\lambda)\partial^i$. If $d$ is a conformal derivation, it needs to satisfy $d_\lambda[L_\mu L]=[(d_\lambda L)_{\lambda+\mu}L]+[L_\mu(d_\lambda L)]$, this is equivalent to saying that
\begin{eqnarray}
(\partial+\lambda+2\mu)q(\partial,\lambda)G=q(-\lambda-\mu,\lambda)(\frac{1}{2}\partial+\frac{3}{2}\lambda
+\frac{3}{2}\mu)G+q(\partial+\mu,\lambda)(\partial+\frac{3}{2}\mu)G.\label{exp2-103}
\end{eqnarray}
Equating terms of degree $m$ in $\partial$ in both sides, we get, if $m>1: (\lambda+\frac{1}{2}\mu-m\mu)b_m(\lambda)=0$, which shows that $b_m(\lambda)=0$.
So $q(\partial,\lambda)=b_0(\lambda)+b_1(\lambda)\partial$, substituting into Eq.(\ref{exp2-103}), we get $b_0(\lambda)=3\lambda b_1(\lambda)$, hence $d_\lambda L=b_1(\lambda)(\partial+3\lambda)G$.

Similarly, $d_\lambda[G_\mu G]=[(d_\lambda G)_{\lambda+\mu}G]-[G_\mu(d_\lambda G)]$, it is equivalent to
\begin{eqnarray}
q(\partial,\lambda)G=f(-\lambda-\mu,\lambda)(\partial+\frac{3}{2}\lambda
+\frac{3}{2}\mu)G-f(\partial+\mu,\lambda)(\frac{1}{2}\partial+\frac{3}{2}\mu)G.\label{exp2-104}
\end{eqnarray}
The degree of $\partial$ in $q(\partial,\lambda)$ is at most one, so $f(\partial,\lambda)=c_0(\lambda)$, substituting into Eq.(\ref{exp2-104}), we get $c_0(\lambda)=2b_1(\lambda)$, hence $d_\lambda G=2b_1(\lambda)L$.

So $d_\lambda=2b_1(\lambda){\rm ad}G_\lambda$. Then every conformal derivation of Neveu-Schwarz Lie conformal superalgebra is an inner conformal derivation.
\end{proof}
\eeg

\begin{defn}\rm \label{def2-9}
Let $\R$ be a Lie conformal superalgebra with $\lambda$-bracket $[-_{\lambda}-]$, $M$ a conformal module of $\R$ and $\rho:\R\rightarrow\mathbb{C}[\lambda]\otimes\Cend(M):r\mapsto\rho(r)_\lambda$ the corresponding representation. A conformal linear map $d\in\Chom(\R,M)$ is a conformal derivation from $\R$ to $M$ if
\begin{eqnarray}
d_\lambda([r_\mu r'])=(-1)^{|r||d|}\rho(r)_\mu(d_\lambda(r'))-(-1)^{(|r|+|d|)|r'|}\rho(r')_{-\partial-\lambda-\mu}(d_\lambda(r)).\label{defb}
\end{eqnarray}
\end{defn}
Note that if $\rho$ is the adjoint representation (that is $\rho(a)_\lambda b=[a_\lambda b]$) of $\R$, then Eq.(\ref{defb}) is Eq.(\ref{defa}). We have the following:
\begin{lem}
Let $\R$ be a Lie conformal superalgebra with $\lambda$-bracket $[-_{\lambda}-]$, $M$ a conformal module of $\R$ and $\rho:\R\rightarrow\mathbb{C}[\lambda]\otimes\Cend(M):r\mapsto\rho(r)_\lambda$ the corresponding representation. For any $m\in M$, define a $\mathbb{C}$-linear map
\begin{eqnarray}
d^m:\R\rightarrow \mathbb{C}[\lambda]\otimes M: d^m_\lambda(r)=-(-1)^{|r||m|}\rho(r)_{-\partial-\lambda}(m), r\in\R.\label{RtoM4}
\end{eqnarray}
Then $d^m$ is a conformal derivation from $\R$ to $M$.
\end{lem}
\begin{proof}
For any $r\in\R, m\in M_{\theta},$ according to Definition \ref{def2.6} and Eq.(\ref{RtoM4}), we have
\begin{eqnarray*}
d^m_\lambda(\partial r)=-(-1)^{|r||m|}\rho(\partial r)_{-\partial-\lambda}(m)
=-(-1)^{|r||m|}(\partial+\lambda)\rho(r)_{-\partial-\lambda}(m)=(\partial+\lambda)d^m_\lambda(r),
\end{eqnarray*}
and $d^m(\R_{\eta})\subseteq M_{\eta+\theta}$, which means that $d^m\in \Cend(\R,M)_{\theta}$. For any $r'\in\R,$
\begin{eqnarray}
-d^m_\lambda([r_{\mu}r'])
&=&-\sum_{s\geq0}\frac{\mu^s}{s!}d^m_\lambda(r_{(s)}r')\nonumber\\
&=&(-1)^{(|r|+|r'|)|m|}\sum_{s\geq0}\frac{\mu^s}{s!}\rho(r_{(s)}r')_{-\partial-\lambda}(m)\nonumber\\
&=&(-1)^{(|r|+|r'|)|m|}\sum_{s,t\geq0}\frac{\mu^s}{s!}\frac{(-\partial-\lambda)^t}{t!}\rho(r_{(s)}r')_{(t)}(m).\label{RtoM1}
\end{eqnarray}

Since $\rho(r)\in\Cend(M)_{|r|}$ is a conformal linear map, it follows from
\begin{eqnarray*}
\rho(r)_\mu((-\partial-\lambda)^t m')=(-\partial-\lambda-\mu)^t\rho(r)_\mu(m'), m'\in M,
\end{eqnarray*}
and hence,
\begin{eqnarray}
&&-(-1)^{|r||m|}\rho(r)_\mu(d^m_\lambda(r'))\nonumber\\
&=&(-1)^{|r||m|}(-1)^{|r'||m|}\rho(r)_\mu(\rho(r')_{-\partial-\lambda}(m))\nonumber\\
&=&(-1)^{(|r|+|r'|)|m|}\rho(r)_\mu(\sum_{t\geq0}\frac{(-\partial-\lambda)^t}{t!}\rho(r')_{(t)}(m))\nonumber\\
&=&(-1)^{(|r|+|r'|)|m|}\sum_{s,t\geq0}\frac{\mu^s}{s!}\frac{(-\partial-\lambda-\mu)^t}{t!}\rho(r)_{(s)}(\rho(r')_{(t)}(m)).\label{RtoM2}
\end{eqnarray}

Similarly, $\rho(r')\in\Cend(M)_{|r'|}$ is a conformal linear map, it follows from
\begin{eqnarray*}
\rho(r')_{-\partial-\lambda-\mu}((-\partial-\lambda)^t m')=\mu^t\rho(r')_{-\partial-\lambda-\mu}(m'),
\end{eqnarray*}
and hence,
\begin{eqnarray}
&&-(-1)^{|r'|(|r|+|m|)}\rho(r')_{-\partial-\lambda-\mu}(d^m_\lambda(r))\nonumber\\
&=&(-1)^{|r'|(|r|+|m|)+|m||r|}\rho(r')_{-\partial-\lambda-\mu}(\rho(r)_{-\partial-\lambda}(m))\nonumber\\
&=&(-1)^{|r'|(|r|+|m|)+|m||r|}\sum_{s,t\geq0}\frac{\mu^s}{s!}\frac{(-\partial-\lambda-\mu)^t}{t!}\rho(r')_{(t)}(\rho(r)_{(s)}(m)).\label{RtoM3}
\end{eqnarray}

Since $\rho$ is a representation, by Definition \ref{def2.6}, Eqs.\eqref{RtoM1}, \eqref{RtoM2} and \eqref{RtoM3}, $d^m$ is a conformal derivation from $\R$ to $M$.
\end{proof}

\begin{defn}\rm
Let $\R$ be a Lie conformal superalgebra with $\lambda$-bracket $[-_{\lambda}-]$, $M$ a conformal module of $\R$ and $\rho:\R\rightarrow\mathbb{C}[\lambda]\otimes\Cend(M):r\mapsto\rho(r)_\lambda$ the corresponding representation. $d^m$ given by Eq.\eqref{RtoM4}
is called an inner conformal derivation from $\R$ to $M$.
\end{defn}

\begin{re}\rm
If $\rho$ is the adjoint representation of $\R$, then an inner conformal derivation from $\R$ to $\R$ is nothing rather than an inner conformal derivation of $\R$. In fact, for any $r, r'\in\R$, by Definition \ref{def2-1} $(\rm C2)_{\lambda}$ it follows from
\begin{eqnarray}
d^r_\lambda(r')=-(-1)^{|r||r'|}\rho(r')_{-\partial-\lambda}(r)=-(-1)^{|r||r'|}[r'_{-\partial-\lambda}r]=[r_{\lambda}r'].
\end{eqnarray}
\end{re}

\begin{lem}\label{lemm2-13}
If $f,g\in\Cend(M)$, then for any $m\in M$,
\begin{enumerate}
\item [$(1)$] $f_{\lambda}(g_{-\partial-\mu}m)=(f_{\lambda}g)_{-\partial-\mu}m$,
\item [$(2)$] $f_{-\partial-\lambda}(g_{\mu}m)=(f_{-\partial-\mu}g)_{-\partial-\lambda+\mu}m$,
\item [$(3)$] $f_{-\partial-\lambda}(g_{-\partial-\mu}m)=(f_{-\partial+\mu-\lambda}g)_{-\partial-\mu}m$.
\end{enumerate}
\end{lem}

\begin{lem}
Let $\R$ be a Lie conformal superalgebra with $\lambda$-bracket $[-_{\lambda}-]$, $M$ a conformal module of $\R$ and $\rho:\R\rightarrow\mathbb{C}[\lambda]\otimes\Cend(M):r\mapsto\rho(r)_\lambda$ the corresponding representation. Define a $\lambda$-bracket $[-_\lambda -]_M$ on $\R \oplus M=(\R \oplus M)_{\bar{0}}\oplus(\R \oplus M)_{\bar{1}}$ by
\begin{eqnarray}\label{lemm2-14}
[(r+m)_\lambda(r'+m')]_M=[r_\lambda r']+\rho(r)_\lambda m'-(-1)^{|r'||m|}\rho(r')_{-\partial-\lambda}m,
\end{eqnarray}
$\forall \, r+m,r'+m'\in\R\oplus M$, where $(\R \oplus M)_{\theta}=\R_\theta \oplus M_\theta, \theta\in\mathbb{Z}_{2}.$
Then $\R\oplus M$ is a Lie conformal superalgebra, called the semidirect product of $\R$ and $M$, and denote by $\R\ltimes M$.
\end{lem}
\begin{proof}
$\forall\, r+m,r'+m',r''+m''\in \R\oplus M$, note that $\R \oplus M$ is equipped with a $\mathbb{C}[\partial]$-module structure via
\begin{eqnarray*}
\partial(r+m)=\partial(r)+\partial(m).
\end{eqnarray*}

A direct computation shows that
\begin{eqnarray*}
[\partial(r+m)_\lambda(r'+m')]_M
&=&[(\partial r+\partial m)_\lambda(r'+m')]_M\\
&=&[(\partial r)_\lambda r']+\rho(\partial r)_\lambda(m')-(-1)^{|r'||m|}\rho(r')_{-\partial-\lambda}(\partial m)\\
&=&-\lambda[r_\lambda r']-\lambda \rho(r)_\lambda(m')-(-1)^{|r'||m|}(\partial-\lambda-\partial)\rho(r')_{-\partial-\lambda}(m)\\
&=&-\lambda([r_\lambda r']+\rho(r)_\lambda(m')-(-1)^{|r'||m|}\rho(r')_{-\partial-\lambda}(m))\\
&=&-\lambda[(r+m)_\lambda(r'+m')]_M.
\end{eqnarray*}
Thus $\rm (C1)_\lambda$ holds. $\rm (C2)_\lambda$ follows from
\begin{eqnarray*}
[(r'+m')_{-\partial-\lambda}(r+m)]_M
&=&[r'_{-\partial-\lambda}r]+\rho(r')_{-\partial-\lambda}(m)-(-1)^{|r'||m|}\rho(r)_\lambda(m')\\
&=&-(-1)^{|r'||m|}([r_\lambda r']+\rho(r)_\lambda(m')-(-1)^{|r'||m|}\rho(r')_{-\partial-\lambda}(m))\\
&=&-(-1)^{|r'||m|}[(r+m)_\lambda(r'+m')]_M.
\end{eqnarray*}

To check the Jacobi identity, we compute
\begin{eqnarray}
&&[(r+m)_\lambda[(r'+m')_\mu(r''+m'')]_M]_M\nonumber\\
&=&[(r+m)_\lambda([r'_\mu r'']+\rho(r')_\mu(m'')-(-1)^{|r''||m'|}\rho(r'')_{-\partial-\mu}(m')]_M\nonumber\\
&=&[r_\lambda[r'_\mu r'']]+\rho(r)_\lambda(\rho(r')_\mu(m''))-(-1)^{|r''||m'|}\rho(r)_\lambda(\rho(r'')_{-\partial-\mu}(m'))\nonumber\\
&&-(-1)^{(|r''|+|r'|)|m|}\rho([r'_\mu r''])_{-\partial-\lambda}(m),\label{1}\\[4pt]
&&(-1)^{|r||r'|}[(r'+m')_\mu[(r+m)_\lambda(r''+m'')]_M]_M\nonumber\\
&=&(-1)^{|r||r'|}[r'_\mu[r_\lambda r'']]+(-1)^{|r||r'|}\rho(r')_\mu(\rho(r)_\lambda(m''))\nonumber\\
&&-(-1)^{|r||r'|+|r''||m|}\rho(r')_\mu(\rho(r'')_{-\partial-\lambda}(m))
-(-1)^{|r''||m'|}\rho([r_\lambda r''])_{-\partial-\mu}(m')\label{2}
\end{eqnarray}
and
\begin{eqnarray}
&&[{[(r+m)_\lambda(r'+m')]_M}_{(\lambda+\mu)}(r''+m'')]_M\nonumber\\
&=&[([r_\lambda r']+\rho(r)_\lambda(m')-(-1)^{|r'||m|}\rho(r')_{-\partial-\lambda}(m))_{\lambda+\mu}(r''+m'')]_M\nonumber\\
&=&[[r_\lambda r']_{\lambda+\mu}r'']+\rho([r_\lambda r'])_{\lambda+\mu}(m'')-(-1)^{|r''|(|r|+|m'|)}\rho(r'')_{-\partial-\lambda-\mu}(\rho(r)_\lambda m')\nonumber\\
&&+(-1)^{|r''|(|r'|+|m|)+|r'||m|}\rho(r'')_{-\partial-\lambda-\mu}(\rho(r')_{-\partial-\lambda}m).\label{3}
\end{eqnarray}
By Eqs.\eqref{1}--\eqref{3}, we only need to show that
\begin{eqnarray}
\rho(r)_\lambda (\rho(r')_{-\partial-\mu^{'}}m'')-(-1)^{|r||r'|}\rho(r')_{-\partial-\mu^{'}-\lambda}(\rho(r)_{-\partial-\mu^{'}}m'')
=\rho([r_\lambda r'])_{-\partial-\mu^{'}}(m'').\label{represen3}
\end{eqnarray}

Since $(\rho,M)$ is a representation of $\R$,
\begin{eqnarray}
\rho(r)_\lambda (\rho(r')_\mu m'')-(-1)^{|r||r'|}\rho(r')_\mu (\rho(r)_\lambda m'')=\rho([r_\lambda r'])_{\lambda+\mu}(m'').\label{represen1}
\end{eqnarray}
Replacing $\mu$ by $-\lambda-\mu^{'}-\partial$ in Eq.(\ref{represen1}) and using $(\rm C1)_{\lambda}$, we obtain
\begin{eqnarray}
\rho(r)_\lambda (\rho(r')_{-\partial-\mu'}m'')-(-1)^{|r||r'|}\rho(r')_{-\partial-\mu'-\lambda}(\rho(r)_{\lambda}m'')
=\rho([r_\lambda r'])_{-\partial-\mu'}(m'').\label{represen2}
\end{eqnarray}
By  $(\rm C1)_{\lambda}$ again, Eq.(\ref{represen2}) is equivalent to Eq.\eqref{represen3}. This implies
\begin{eqnarray*}
&&[(r+m)_\lambda[(r'+m')_\mu(r''+m'')]_M]_M\\
&=&(-1)^{|r||r'|}[(r'+m')_\mu[(r+m)_\lambda(r''+m'')]_M]_M+
[{[(r+m)_\lambda(r'+m')]_M}_{(\lambda+\mu)}(r''+m'')]_M.
\end{eqnarray*}
Then $\R\oplus M$ is a Lie conformal superalgebra.
\end{proof}

Let $\g$ be a complex Lie superalgebra and $M$ a finite dimensional $\g$-module. Then we have the Current Lie conformal superalgebra $\rm Cur(\g\ltimes M)$ (see Example \ref{exp2-3}). By Example \ref{exp2-7}, $\mathbb{C}[\partial]\otimes M$ is a conformal module of $\rm Cur\g$. Thus, by the above lemma we have the Lie conformal superalgebra $\rm Cur\g\ltimes(\mathbb{C}[\partial]\otimes M)$.

\begin{cor}
$\rm Cur(\g\ltimes M)$ is a proper Lie conformal subalgebra of $\rm Cur\g\ltimes(\mathbb{C}[\partial]\otimes M)$.
\end{cor}
\begin{proof}
Let $\rho:\g\rightarrow \mathbb{C}[\lambda] \otimes\rm{End_{\mathbb{C}}(M)}$ be the corresponding representation. For any $f(\partial)\otimes(r+m), g(\partial)\otimes(r'+m')\in\rm{Cur(\g\ltimes M)}$, where $r,r'\in\g, m,m'\in M$. By the $\lambda$-bracket on $\rm Cur(\g\ltimes M)$(Example \ref{exp2-3}) it follows from
\begin{eqnarray}
&&[(f(\partial)\otimes(r+m))_\lambda(g(\partial)\otimes(r'+m'))]\nonumber\\
&=&f(-\lambda)g(\partial+\lambda)\otimes[r+m,r'+m']\nonumber\\
&=&f(-\lambda)g(\partial+\lambda)\otimes([r,r']+\rho(r)(m')-(-1)^{|r'||m|}\rho(r')(m)).\label{mod1}
\end{eqnarray}

By the module structure on $\mathbb{C}[\partial]\otimes M$(see Example \ref{exp2-7}) it follows from
\begin{eqnarray*}
\rho(g(\partial)\otimes r')_{-\partial-\lambda}(f(\partial)\otimes m)
=g(\partial+\lambda)f(-\lambda)\otimes\rho(r')(m)
=f(-\lambda)g(\partial+\lambda)\otimes\rho(r')(m).
\end{eqnarray*}
So, by the $\lambda$-bracket on $\rm Cur\g\ltimes(\mathbb{C}[\partial]\otimes M)$, we have
\begin{eqnarray}
&&[(f(\partial)\otimes(r+m))_\lambda(g(\partial)\otimes(r'+m'))]\nonumber\\
&=&[(f(\partial)\otimes r+f(\partial)\otimes m))_\lambda(g(\partial)\otimes r'+g(\partial)\otimes m')]\nonumber\\
&=&[(f(\partial)\otimes r)_\lambda(g(\partial)\otimes r')]+\rho(f(\partial)\otimes r)_\lambda(g(\partial)\otimes m'))\nonumber\\
&&-(-1)^{|r'||m|}\rho(g(\partial)\otimes r')_{-\partial-\lambda}(f(\partial)\otimes m))\nonumber\\
&=&f(-\lambda)g(\partial+\lambda)\otimes[r,r']+f(-\lambda)g(\partial+\lambda)\otimes\rho(r)(m')\nonumber\\
&&-(-1)^{|r'||m|}f(-\lambda)g(\partial+\lambda)\otimes\rho(r')(m)\nonumber\\
&=&f(-\lambda)g(\partial+\lambda)\otimes([r,r']+\rho(r)(m')-(-1)^{|r'||m|}\rho(r')(m)).\label{mod2}
\end{eqnarray}
By Eqs.\eqref{mod1} and \eqref{mod2}, the $\lambda$-bracket on $\rm Cur(\g\ltimes M)$ is induced by that on $\rm Cur\g\ltimes(\mathbb{C}[\partial]\otimes M)$.
\end{proof}

\begin{lem}\label{lemm2-16}
Let $\R$ be a Lie conformal superalgebra with $\lambda$-bracket $[-_\lambda-]$, $M$ a $\mathbb{C}[\partial]$-module, $\rho\in\Chom(\R,\Cend(M))$, and $f\in\Chom(M,\R)$. Then for any $r\in\R$ and $m\in M$, the following two equations are equivalent:
\begin{enumerate}
\item [$(1)$]  $f_\lambda(\rho(r)_\mu(m))=(-1)^{|r||f|}[r_\mu(f_\lambda(m))]$.
\item [$(2)$]  $f_\lambda(\rho(r)_{-\partial-\mu}(m))=-(-1)^{|r||m|}[(f_\lambda(m))_{\lambda+\mu}r]$.
\end{enumerate}
\end{lem}
\begin{proof} If
$ f_\lambda(\rho(r)_\mu(m))=(-1)^{|r||f|}[r_\mu(f_\lambda(m))],$
replacing $\mu$ by $-\lambda-\mu^{'}-\partial$ and using $(\rm C1)_{\lambda}$, we obtain
\begin{eqnarray*}
f_\lambda(\rho(r)_{-\partial-\mu'}(m))=-(-1)^{|r||m|}[(f_\lambda(m))_{\lambda+\mu'}r],
\end{eqnarray*}
that is $(2)$ holds. The reverse conclusion follows similarly.
\end{proof}

\begin{lem}\label{lemm2-17}
Let $\R$ be a Lie conformal superalgebra with $\lambda$-bracket $[-_\lambda-]$, $M$ a $\mathbb{C}[\partial]$-module, $\rho\in\Chom(\R,\Cend(M))$, $f\in\Cend(M)_\theta$ and $g\in\Cend(\R)_\theta$. Then for any $r\in\R$ and $m\in M$, the following two equations are equivalent:
\begin{enumerate}
\item [$(1)$]  $f_\lambda(\rho(r)_\mu(m))=\rho(g_\lambda(r))_{\lambda+\mu}(m)+(-1)^{\theta|r|}\rho(r)_\mu(f_\lambda(m))$.
\item [$(2)$]  $f_\lambda(\rho(r)_{-\partial-\mu}(m))=\rho(g_\lambda(r))_{-\partial-\mu}(m)
    +(-1)^{\theta|r|}\rho(r)_{-\partial-\lambda-\mu}(f_\lambda(m))$.
\end{enumerate}
\end{lem}
\begin{proof}
Its proof is similar to Lemma \ref{lemm2-16}.
\end{proof}

Let $\R$ be a Lie conformal superalgebra, and $M$ a conformal module of $\R$. Hereafter we denote $\mathbb{C}$-linear maps from the $\mathbb{C}[\partial]$-module $\R\oplus M$ to $\mathbb{C}[\lambda]\otimes(\R\oplus M)$ in the following form of matrices:
\begin{displaymath}
(*) \ \ \ \  d_\lambda =
\left( \begin{array}{cc}
\ d_{11} & d_{12} \\
d_{21} & d_{22}
\end{array}\right)_\lambda:\R \oplus M\rightarrow \mathbb{C}[\lambda]\otimes(\R\oplus M), where
\end{displaymath}
\begin{eqnarray*}
&{d_{11}}_\lambda:\R\rightarrow \mathbb{C}[\lambda]\otimes\R, {d_{12}}_\lambda:M\rightarrow \mathbb{C}[\lambda]\otimes\R,\\
&{d_{21}}_\lambda:\R\rightarrow \mathbb{C}[\lambda]\otimes M,{d_{22}}_\lambda:M\rightarrow \mathbb{C}[\lambda]\otimes M.
\end{eqnarray*}

\begin{lem}\label{lemm2-18}
Let $\R$ be a Lie conformal superalgebra, and $M$ a conformal module of $\R$. Then $d$ is given by $(*)$ is a conformal linear map of degree $\theta$ if and only if all
$d_{11},d_{12},d_{21}$ and $d_{22}$ are conformal linear maps of degree $\theta$.
\end{lem}
\begin{proof}
It is straightforward by Definition \ref{def2-4}.
\end{proof}

\begin{thm}\label{th2-19}
Let $\R$ be a Lie conformal superalgebra with $\lambda$-bracket$[-_\lambda-]$, $M$ a conformal module of $\R$, and $\rho:\R\rightarrow\mathbb{C}[\lambda]\otimes\Cend(M):r\rightarrow\rho(r)_\lambda$ the corresponding representation. Then a conformal linear map $d_\lambda$ given by $(*)$ is a conformal derivation of degree $\theta$ of $\R\ltimes M$ if and only if $d_\lambda$ satisfies the following conditions:
\begin{enumerate}
\item[$(1)$] $d_{11}$ is a conformal derivation of degree $\theta$ of $\R$.
\item[$(2)$] For any $r\in\R$ and $m,m'\in M$, ${d_{12}}_\lambda(\rho(r)_\mu m)=(-1)^{\theta|r|}[r_\mu({d_{12}}_\lambda m)],$\\
$\rho({d_{12}}_\lambda m)_{\lambda+\mu}(m')=(-1)^{|m||m'|}\rho({d_{12}}_\lambda m')_{-\partial-\mu}(m)$.
\item[$(3)$] $d_{21}$ is a conformal derivation of degree $\theta$ from $\R$ to $M$.
\item[$(4)$] For any $r\in\R$ and $m\in M$,\\
${d_{22}}_\lambda(\rho(r)_\mu m)=\rho({d_{11}}_\lambda r)_{\lambda+\mu}(m)+(-1)^{\theta|r|}\rho(r)_\mu({d_{22}}_\lambda m).$
\end{enumerate}
\end{thm}
\begin{proof}
For any $r+m,r'+m'\in\R\ltimes M$, by Eq.(\ref{lemm2-14}), we have
\begin{eqnarray*}
&&d_\lambda([(r+m)_\mu(r'+m')])\\
&=&{d_{11}}_\lambda([r_\mu r'])+{d_{12}}_\lambda(\rho(r)_\mu m')-(-1)^{|r'||m|}{d_{12}}_\lambda(\rho(r')_{-\partial-\mu}m)\\
&&+{d_{21}}_\lambda([r_\mu r'])+{d_{22}}_\lambda(\rho(r)_\mu m')-(-1)^{|r'||m|}{d_{22}}_\lambda(\rho(r')_{-\partial-\mu}m),\\
&&[(d_\lambda(r+m))_{\lambda+\mu}(r'+m')]\\
&=&[({d_{11}}_\lambda(r)+{d_{12}}_\lambda(m)+{d_{21}}_\lambda(r)+{d_{22}}_\lambda(m))_{\lambda+\mu}(r'+m')]\\
&=&[({d_{11}}_\lambda(r))_{\lambda+\mu}r']+[({d_{12}}_\lambda(m))_{\lambda+\mu}r']+\rho({d_{11}}_\lambda(r))_{\lambda+\mu}(m')
+\rho({d_{12}}_\lambda(m))_{\lambda+\mu}(m')\\
&&-(-1)^{|r'|(|r|+\theta)}\rho(r')_{-\partial-\lambda-\mu}({d_{21}}_\lambda(r))
-(-1)^{|r'|(|m|+\theta)}\rho(r')_{-\partial-\lambda-\mu}({d_{22}}_\lambda(m))
\end{eqnarray*}
and
\begin{eqnarray*}
&&(-1)^{\theta|r|}[(r+m)_\mu (d_\lambda(r'+m'))]\\
&=&(-1)^{\theta|r|}[(r+m)_\mu({d_{11}}_\lambda(r')+{d_{12}}_\lambda(m')+{d_{21}}_\lambda(r')+{d_{22}}_\lambda(m'))]\\
&=&(-1)^{\theta|r|}[r_\mu({d_{11}}_\lambda(r'))]+(-1)^{\theta|r|}[r_\mu({d_{12}}_\lambda(m'))]+(-1)^{\theta|r|}\rho(r)_\mu({d_{21}}_\lambda(r'))\\
&&+(-1)^{\theta|r|}\rho(r)_\mu({d_{22}}_\lambda(m'))-(-1)^{|r||r'|}\rho({d_{11}}_\lambda(r'))_{-\partial-\mu}(m)\\
&&-(-1)^{|m||m'|}\rho({d_{12}}_\lambda(m'))_{-\partial-\mu}(m).
\end{eqnarray*}

Suppose that $d_\lambda$ is a conformal derivation of degree $\theta$ of $\R\ltimes M$. By Eq.\eqref{defa} and taking $m=m'=0$ in the above identities, we get
\begin{eqnarray}
&&{d_{11}}_\lambda[r_\mu r']=[({d_{11}}_\lambda(r))_{\lambda+\mu}r']+(-1)^{\theta|r|}[r_\mu({d_{11}}_\lambda(r'))],\label{2-20}\\
&&{d_{21}}_\lambda[r_\mu r']=(-1)^{\theta|r|}\rho(r)_\mu({d_{21}}_\lambda(r'))-(-1)^{|r'|(|r|+\theta)}\rho(r')_{-\partial-\lambda-\mu}({d_{21}}_\lambda(r));\label{2-21}
\end{eqnarray}

Taking $r=r'=0$,
\begin{eqnarray}
\rho({d_{12}}_\lambda(m))_{\lambda+\mu}(m')-(-1)^{|m||m'|}\rho({d_{12}}_\lambda(m'))_{-\partial-\mu}(m)=0;\label{2-22}
\end{eqnarray}

Taking $m=0,r'=0$,
\begin{eqnarray}
&&{d_{12}}_\lambda(\rho(r)_\mu(m'))=(-1)^{\theta|r|}[r_\mu({d_{12}}_\lambda(m'))],\label{2-23}\\
&&{d_{22}}_\lambda(\rho(r)_\mu(m'))=\rho({d_{11}}_\lambda(r))_{\lambda+\mu}(m')+(-1)^{\theta|r|}\rho(r)_\mu({d_{22}}_\lambda(m'));\label{2-24}
\end{eqnarray}

Taking $m'=0,r=0$,
\begin{eqnarray}
&&-(-1)^{|r'||m|}{d_{12}}_\lambda(\rho(r')_{-\partial-\mu}m)=[({d_{12}}_\lambda(m))_{\lambda+\mu}r'],\label{2-25}\\
&&{d_{22}}_\lambda(\rho(r')_{-\partial-\mu}m)=\rho({d_{11}}_\lambda(r'))_{-\partial-\mu}(m)+(-1)^{\theta|r'|}\rho(r')_{-\partial-\lambda-\mu}({d_{22}}_\lambda(m)).\label{2-26}
\end{eqnarray}
By Lemmas \ref{lemm2-16} and \ref{lemm2-17}, Eq.\eqref{2-23} is equivalent to Eq.\eqref{2-25}, and Eq.\eqref{2-24} is equivalent to Eq.\eqref{2-26}.
By Definition \ref{def2-9}, Lemmas \ref{lemm2-16}, \ref{lemm2-17} and \ref{lemm2-18}, $d_\lambda$ is a conformal derivation of degree $\theta$ of $\R\ltimes M$ if and only if Eqs.\eqref{2-20}-\eqref{2-26} are satisfied. Eqs.\eqref{2-20}-\eqref{2-26} are equivalent to conditions $(1)-(4)$ of Theorem \ref{th2-19}.
\end{proof}

\section{Deformations of Lie conformal superalgebras}

In the following we aim to develop cohomology theory of Lie conformal superalgebras. To do this, we need the following concept.

\bdefn\rm
An $n$-cochain ($n\in \mathbb{Z}_+$) of a Lie conformal superalgebra $\R$ with coefficients in a module $M$ is a $\mathbb{C}$-linear map of degree $\theta$
\begin{eqnarray*}
\gamma:\R^{n}&\rightarrow& M[\lambda_{1},\cdots,\lambda_{n}],\\
(a_1,\cdots,a_n)&\mapsto& \gamma_{\lambda_{1},\cdots,\lambda_{n}}(a_1,\cdots,a_n),
\end{eqnarray*}
where $M[\lambda_{1},\cdots,\lambda_{n}]$ denotes the space of polynomials with coefficients in $M$, satisfying the following conditions:

Conformal antilinearity: $$\gamma_{\lambda_{1},\cdots,\lambda_{n}}(a_1,\cdots,\partial a_i,\cdots,a_n) =-\lambda_{i}\gamma_{\lambda_{1},\cdots,\lambda_{n}}(a_1,\cdots,a_i,\cdots,a_n);$$

Skew-symmetry:
\begin{eqnarray*}&&\gamma_{\lambda_{1},\cdots,\lambda_i,\cdots,\lambda_j,\cdots,\lambda_{n}}(a_1,\cdots,a_i,\cdots,a_j,\cdots,a_n)\\
&=&-(-1)^{|a_i||a_j|}\gamma_{\lambda_{1},\cdots,\lambda_j,\cdots,\lambda_i,\cdots,\lambda_{n}}(a_1,\cdots,a_j,\cdots,a_i,\cdots,a_n).\end{eqnarray*}
\edefn

Let $\R^{\otimes0}=\mathbb{C}$ as usual, so that a $0$-cochain $\gamma$ is an element of $M$. Define a differential ${{\rm{\bf d}}}$ of a cochain $\gamma$ by
\begin{eqnarray*}
&&({{\rm{\bf d}}}\gamma)_{\lambda_{1},\cdots,\lambda_{n+1}}(a_1,\cdots,a_{n+1})\\
&=&\mbox{$\sum\limits_{i=1}^{n+1}$}(-1)^{i+1}(-1)^{(|\gamma|+|a_1|+\cdots+|a_{i-1}|)|a_i|}\rho(a_i)_{\lambda_i}\gamma_{\lambda_{1},\cdots,\hat{\lambda_i},
\cdots,\lambda_{n+1}}(a_1,\cdots,\hat{a_i},\cdots,a_{n+1})\\
&&+\mbox{$\sum\limits_{1\leq i<j}^{n+1}$}(-1)^{i+j}(-1)^{(|a_1|+\cdots+|a_{i-1}|)|a_i|+(|a_1|+\cdots+|a_{j-1}|)|a_j|+|a_i||a_j|}\\
&&\gamma_{\lambda_i+\lambda_j,\lambda_1,\cdots,\hat{\lambda_i},\cdots,\hat{\lambda}_j,\cdots,\lambda_{n+1}}
([{a_i}_{\lambda_i}a_j],a_1,\cdots,\hat{a}_i,\cdots,\hat{a}_j,\cdots,a_{n+1}),
\end{eqnarray*}
where $\rho$ is the corresponding representation of $M$, and $\gamma$ is extended linearly over the polynomials in $\lambda_i$. In particular, if $\gamma$ is a $0$-cochain, then $({{\rm{\bf d}}}\gamma)_\lambda a=a_\lambda \gamma$.
\begin{re}\rm
Conformal antilinearity implies the following relation for an $n$-cochain $\gamma$:
$$\gamma_{\lambda+\mu,\lambda_1,\cdots}([a_\lambda b],a_1,\cdots)=\gamma_{\lambda+\mu,\lambda_1,\cdots}([a_{-\partial-\mu} b],a_1,\cdots).$$
\end{re}
\begin{prop} ${{\rm{\bf d}}}\gamma$ is a cochain and ${{\rm{\bf d}}}^{2}=0$.
\end{prop}
\begin{proof} Let $\gamma$ be an $n$-cochain. As discussed in the proof of \cite[Lemma 2.1]{BKV},
${{\rm{\bf d}}}\gamma$ satisfies conformal antilinearity and skew-symmetry. Thus
${{\rm{\bf d}}}\gamma$ is an $(n+1)$-cochain.

A straightforward computation shows that
\begin{align}
&({{\rm{\bf d}}}^{2}\gamma)_{\lambda_{1},\cdots,\lambda_{n+2}}(a_1,\cdots,a_{n+2})\notag\\
=&\sum_{i=1}^{n+2}(-1)^{i+1+|\gamma||a_i|+A_i}\rho(a_i)_{\lambda_i}({{\rm{\bf d}}}\gamma)_{\lambda_{1},\cdots,\hat{\lambda}_i,
\cdots,\lambda_{n+2}}(a_1,\cdots,\hat{a}_i,\cdots,a_{n+2})\notag\\
&+\sum_{1\leq i<j}^{n+2}(-1)^{i+j+A_i+A_j+|a_i||a_j|}({{\rm{\bf d}}}\gamma)_{\lambda_i+\lambda_j,\lambda_1,\cdots,\hat{\lambda}_{i,j},\cdots,\lambda_{n+2}}
([{a_i}_{\lambda_i}a_j],a_1,\cdots,\hat{a}_{i,j},\cdots,a_{n+2})\notag\\
=&\sum_{1\leq j<i}^{n+2}(-1)^{i+j+|\gamma|(|a_i|+|a_j|)+A_i+A_j}
\rho(a_i)_{\lambda_i}(\rho(a_j)_{\lambda_j}\gamma_{\lambda_{1},\cdots,\hat{\lambda}_{j,i},
\cdots,\lambda_{n+2}}(a_1,\cdots,\hat{a}_{j,i},\cdots,a_{n+2}))\label{*1}\\
&+\sum_{1\leq i<j}^{n+2}(-1)^{i+j+1+|\gamma|(|a_i|+|a_j|)+A_i+(A_j-|a_i|)}\notag\\
&\rho(a_i)_{\lambda_i}(\rho(a_j)_{\lambda_j}\gamma_{\lambda_{1},\cdots,\hat{\lambda}_{i,j},
\cdots,\lambda_{n+2}}(a_1,\cdots,\hat{a}_{i,j},\cdots,a_{n+2}))\label{*2}\\
&+\sum_{1\leq j<k<i}^{n+2}(-1)^{i+j+k+1+|\gamma||a_i|+A_i+A_j+A_k+|a_j||a_k|}\notag\\
&\rho(a_i)_{\lambda_i}\gamma_{\lambda_j+\lambda_k,\lambda_1,\cdots,\hat{\lambda}_{j,k,i},\cdots,\lambda_{n+2}}
([{a_j}_{\lambda_j}a_k],a_1,\cdots,\hat{a}_{j,k,i},\cdots,a_{n+2})\label{*3}\\
&+\sum_{1\leq j<i<k}^{n+2}(-1)^{i+j+k+|\gamma||a_i|+A_i+A_j+(A_k-|a_i|)+|a_j||a_k|}\notag\\
&\rho(a_i)_{\lambda_i}\gamma_{\lambda_j+\lambda_k,\lambda_1,\cdots,\hat{\lambda}_{j,i,k},\cdots,\lambda_{n+2}}
([{a_j}_{\lambda_j}a_k],a_1,\cdots,\hat{a}_{j,i,k},\cdots,a_{n+2})\label{*4}\\
&+\sum_{1\leq i<j<k}^{n+2}(-1)^{i+j+k+1+|\gamma||a_i|+A_i+(A_j-|a_i|)+(A_k-|a_i|)+|a_j||a_k|}\notag\\
&\rho(a_i)_{\lambda_i}\gamma_{\lambda_j+\lambda_k,\lambda_1,\cdots,\hat{\lambda}_{i,j,k},\cdots,\lambda_{n+2}}
([{a_j}_{\lambda_j}a_k],a_1,\cdots,\hat{a}_{i,j,k},\cdots,a_{n+2})\label{*5}\\
&+\sum_{1\leq k<i<j}^{n+2}(-1)^{i+j+k+A_i+A_j+A_k+|a_i||a_j|+(|\gamma|+|a_i|+|a_j|)|a_k|}\notag\\
&\rho(a_k)_{\lambda_k}\gamma_{\lambda_i+\lambda_j,\lambda_1,\cdots,\hat{\lambda}_{k,i,j},\cdots,\lambda_{n+2}}
([{a_i}_{\lambda_i}a_j],a_1,\cdots,\hat{a}_{k,i,j},\cdots,a_{n+2})\label{*6}\\
&+\sum_{1\leq i<k<j}^{n+2}(-1)^{i+j+k+1+A_i+A_j+A_k+|a_i||a_j|+(|\gamma|+|a_j|)|a_k|}\notag\\
&\rho(a_k)_{\lambda_k}\gamma_{\lambda_i+\lambda_j,\lambda_1,\cdots,\hat{\lambda}_{i,k,j},\cdots,\lambda_{n+2}}
([{a_i}_{\lambda_i}a_j],a_1,\cdots,\hat{a}_{i,k,j},\cdots,a_{n+2})\label{*7}\\
&+\sum_{1\leq i<j<k}^{n+2}(-1)^{i+j+k+A_i+A_j+A_k+|a_i||a_j|+|\gamma||a_k|}\notag\\
&\rho(a_k)_{\lambda_k}\gamma_{\lambda_i+\lambda_j,\lambda_1,\cdots,\hat{\lambda}_{i,j,k},\cdots,\lambda_{n+2}}
([{a_i}_{\lambda_i}a_j],a_1,\cdots,\hat{a}_{i,j,k},\cdots,a_{n+2})\label{*8}\\
&+\sum_{1\leq i<j}^{n+2}(-1)^{i+j+A_i+A_j+|a_i||a_j|+|\gamma|(|a_i|+|a_j|)}\notag\\
&\rho([{a_i}_{\lambda_i}a_j])_{\lambda_i+\lambda_j}\gamma_{\lambda_{1},\cdots,\hat{\lambda}_j,\cdots,\hat{\lambda}_i,
\cdots,\lambda_{n+2}}(a_1,\cdots,\hat{a}_j,\cdots,\hat{a}_i,\cdots,a_{n+2})\label{*9}\\
&+\sum_{distinct\,i,j,k,l,i<j,k<l}^{n+2}(-1)^{i+j+k+l}sign\{i,j,k,l\}(-1)^{A_i+A_j+|a_i||a_j|+(|a_i|+|a_j|)(|a_k|+|a_l|)}\notag\\
&(-1)^{(A_k-|a_i|-|a_j|)+(A_l-|a_i|-|a_j|)+|a_k||a_l|}\notag\\
&\gamma_{\lambda_k+\lambda_l,\lambda_i+\lambda_j,\lambda_1,\cdots,\hat{\lambda}_{i,j,k,l},\cdots,\lambda_{n+2}}
([{a_k}_{\lambda_k}a_l],[{a_i}_{\lambda_i}a_j],a_1,\cdots,\hat{a}_{i,j,k,l},\cdots,a_{n+2})\label{*10}\\
&+\sum_{i,j,k=1,i<j,k\neq i,j}^{n+2}(-1)^{i+j+k+1}sign\{i,j,k\}(-1)^{A_i+A_j+|a_i||a_j|+(A_k-|a_i|-|a_j|)}\notag\\
&\gamma_{\lambda_i+\lambda_j+\lambda_k,\lambda_1,\cdots,\hat{\lambda}_{i,j,k},\cdots,\lambda_{n+2}}
([[{a_i}_{\lambda_i}a_j]_{\lambda_i+\lambda_j}a_k],a_1,\cdots,\hat{a}_{i,j,k},\cdots,a_{n+2})\label{*11},
\end{align}
where $A_i=(|a_1|+\cdots+|a_{i-1}|)|a_i|, A_j-|a_i|=(|a_1|+\cdots+\hat{a_i}+\cdots+|a_{j-1}|)|a_j|$ and $sign\{i_1,\cdots,i_p\}$ is the sign of the permutation putting the indices in increasing order and $\hat{a}_{i,j,\cdots}$
means that $a_i,a_j,\cdots$ are omitted.

It is obvious that Eqs.\eqref{*3} and \eqref{*8} summations cancel each other. The same is true for Eqs.\eqref{*4} and \eqref{*7}, \eqref{*5} and \eqref{*6}. The Jacobi identity implies Eq.\eqref{*11} $=0$, whereas skew-symmetry of $\gamma$ gives Eq.\eqref{*10} $=0$. As $M$ is an $\R$-module,
\begin{eqnarray*}
-\rho(a_i)_{\lambda_i}(\rho(a_j)_{\lambda_j}m)+(-1)^{|a_i||a_j|}\rho(a_j)_{\lambda_j}(\rho(a_i)_{\lambda_i}m)
+\rho([{a_i}_{\lambda_i}a_j])_{\lambda_i+\lambda_j}(m)=0.
\end{eqnarray*}
 Eqs.\eqref{*1}, \eqref{*2} and \eqref{*9} summations cancel.  This proves ${{\rm{\bf d}}}^{2}\gamma=0$.
\end{proof}

Thus the cochains of a Lie conformal superalgebra $\R$ with coefficients in a module $M$ form a comlex, which is denoted by
 $$\rm C^{\bullet}=C^{\bullet}(\R,M)=\bigoplus_{n\in\Z_{+}}C^{n}(\R,M).$$ Where $\rm C^{n}(\R,M)=C^{n}(\R,M)_{\bar{0}}\oplus C^{n}(\R,M)_{\bar{1}}$, $C^{n}(\R,M)_{\theta}$ is the set of $n$-cochain of degree $\theta$ $(\theta\in\mathbb{Z}_{2})$.
The \emph{cohomology} ${\rm H}^{\bullet}(\R,M)$ of a Lie conformal superalgebra $\R$ with coefficients
in a module $M$ is the cohomology of complex $\rm C^{\bullet}$.

Let $\R$ be a Lie conformal superalgebra. Define
\begin{eqnarray}\label{la2}
\rho(a)_\lambda b=[a_\lambda b],\ \ \forall \ a,b\in\R.
\end{eqnarray}
\begin{prop} $(\rho,\R)$ is a representation with the $\lambda$-action given in {\rm Eq.}\eqref{la2}.
\end{prop}
\begin{proof} It only consists of checking the axioms from Definition \ref{def2.6}.
\end{proof}

Let $\gamma\in C^{n}(\R,\R)$. Define an operator ${\rm{\bf \hat{d}}}:C^{n}(\R,\R)\rightarrow C^{n+1}(\R,\R)$ by
\begin{eqnarray*}
&&({\rm{\bf \hat{d}}}\gamma)_{\lambda_{1},\cdots,\lambda_{n+1}}(a_1,\cdots,a_{n+1})\\
&=&\sum_{i=1}^{n+1}(-1)^{i+1}(-1)^{(|\gamma|+|a_1|+\cdots+|a_{i-1}|)|a_i|}[{a_i}_{\lambda_i}\gamma_{\lambda_{1},\cdots,\hat{\lambda_i},
\cdots,\lambda_{n+1}}(a_1,\cdots,\hat{a_i},\cdots,a_{n+1})]\\
&&+\sum_{1\leq i<j}^{n+1}(-1)^{i+j}(-1)^{(|a_1|+\cdots+|a_{i-1}|)|a_i|+(|a_1|+\cdots+|a_{j-1}|)|a_j|+|a_i||a_j|}\\
&&\gamma_{\lambda_i+\lambda_j,\lambda_1,\cdots,\hat{\lambda_i},\cdots,\hat{\lambda}_j,\cdots,\lambda_{n+1}}
([{a_i}_{\lambda_i}a_j],a_1,\cdots,\hat{a}_i,\cdots,\hat{a}_j,\cdots,a_{n+1}).
\end{eqnarray*}
Obviously, the operator ${\rm{\bf \hat{d}}}$ is induced from the differential ${{\rm{\bf d}}}$. Thus ${\rm{\bf d}}$ preserves the space of cochains and satisfies ${\rm{\bf \hat{d}}}^2=0$. In the following the complex $C^{\bullet}(\R, \R)$ is assumed to be associated with the differential ${\rm{\bf \hat{d}}}$.

For $\psi\in C^{2}(\R,\R)_{\bar{0}}$, we consider a $t$-parameterized family of bilinear operations on $\R$
\begin{eqnarray}\label{1-3}
[a_\lambda b]_t=[a_\lambda b]+t\psi_{\lambda,-\partial-\lambda}(a,b), \ \forall \, a,b\in\R.
\end{eqnarray}
If $[\cdot_{\lambda}\cdot]_{t}$ endows $(\R, [\cdot_{\lambda}\cdot]_{t}, \alpha)$ with a Lie conformal superalgebra structure, we say that $\psi$ generates a deformation of the Lie conformal superalgebra $\R$. It is easy to see that $[\cdot_{\lambda}\cdot]_{t}$ satisfies $(\rm C1)_\lambda$ and $(\rm C2)_\lambda$. If it is true for $(\rm C3)_\lambda$, expanding the Jacobi identity for $[\cdot_{\lambda}\cdot]_{t}$ gives
\begin{eqnarray*}
&&[a_\lambda[b_\mu c]]+t([a_\lambda(\psi_{\mu,-\partial-\mu}(b,c))]
+\psi_{\lambda,-\partial-\lambda}(a,[b_\mu c]))\\
&&+t^{2}\psi_{\lambda,-\partial-\lambda}(a,\psi_{\mu,-\partial-\mu}(b,c))\\
&=&(-1)^{|a||b|}[b_\mu[a_\lambda c]]+(-1)^{|a||b|}t([b_\mu(\psi_{\lambda,-\partial-\lambda}(a,c))]
+\psi_{\mu,-\partial-\mu}(b,[a_\lambda c]))\\
&&+(-1)^{|a||b|}t^{2}\psi_{\mu,-\partial-\mu}(b,\psi_{\lambda,-\partial-\lambda}(a,c))\\
&&+[[a_\lambda b]_{\lambda+\mu}c]+t([(\psi_{\lambda,-\partial-\lambda}(a,b))_{\lambda+\mu}c]
+\psi_{\lambda+\mu,-\partial-\lambda-\mu}([a_\lambda b],c))\\
&&+t^{2}\psi_{\lambda+\mu,-\partial-\lambda-\mu}(\psi_{\lambda,-\partial-\lambda}(a,b),c).
\end{eqnarray*}
This is equivalent to the following relations
\begin{eqnarray}
&&[a_\lambda(\psi_{\mu,-\partial-\mu}(b,c))]+\psi_{\lambda,-\partial-\lambda}(a,[b_\mu c])\notag\\
&=&(-1)^{|a||b|}[b_\mu(\psi_{\lambda,-\partial-\lambda}(a,c))]+(-1)^{|a||b|}\psi_{\mu,-\partial-\mu}(b,[a_\lambda c])\notag\\
&&+[(\psi_{\lambda,-\partial-\lambda}(a,b))_{\lambda+\mu}c]+\psi_{\lambda+\mu,-\partial-\lambda-\mu}([a_\lambda b],c)\label{defor1}
\end{eqnarray}and
\begin{eqnarray}
&&\psi_{\lambda,-\partial-\lambda}(a,\psi_{\mu,-\partial-\mu}(b,c))\notag\\
&=&(-1)^{|a||b|}\psi_{\mu,-\partial-\mu}(b,\psi_{\lambda,-\partial-\lambda}(a,c))
+\psi_{\lambda+\mu,-\partial-\lambda-\mu}(\psi_{\lambda,-\partial-\lambda}(a,b),c).\label{defor2}
\end{eqnarray}
By conformal antilinearity of $\psi$, we have
\begin{eqnarray}
[(\psi_{\lambda,-\partial-\lambda}(a,b))_{\lambda+\mu}c]=[\psi_{\lambda,\mu}(a,b)_{\lambda+\mu}c].\label{defor3}
\end{eqnarray}

On the other hand, let $\psi$ be a cocycle, i.e., ${\rm{\bf \hat{d}}}\psi=0.$ Explicitly,
\begin{eqnarray}
0&=&({\rm{\bf \hat{d}}}\psi)_{\lambda,\mu,\gamma}(a,b,c)\notag\\
&=&[a_\lambda(\psi_{\mu,\gamma}(b,c))]-(-1)^{|a||b|}[b_\mu(\psi_{\lambda,\gamma}(a,c))]
+(-1)^{(|a|+|b|)|c|}[c_{\gamma}(\psi_{\lambda,\mu}(a,b))]\notag\\
&&-\psi_{\lambda+\mu,\gamma}([a_\lambda b],c)+(-1)^{|b||c|}\psi_{\lambda+\gamma,\mu}([a_\lambda c],b)
-(-1)^{|a||b|+|a||c|}\psi_{\mu+\gamma,\lambda}([b_\mu c],a)\notag\\
&=&[a_\lambda(\psi_{\mu,\gamma}(b,c))]-(-1)^{|a||b|}[b_\mu(\psi_{\lambda,\gamma}(a,c))]
-[(\psi_{\lambda,\mu}(a,b))_{-\partial-\gamma}c]\notag\\
&&+\psi_{\lambda,\mu+\gamma}(a,[b_\mu c])-(-1)^{|a||b|}\psi_{\mu,\lambda+\gamma}(b,[a_\lambda c])
-\psi_{\lambda+\mu,\gamma}([a_\lambda b],c). \label{defor4}
\end{eqnarray}
By $(\rm C1)_\lambda$, Eq.\eqref{defor3} and replacing $\gamma$ by $-\lambda-\mu-\partial$ in Eq.(\ref{defor4}), we obtain
\begin{eqnarray*}
0&=&[a_\lambda(\psi_{\mu,-\partial-\mu}(b,c))]-(-1)^{|a||b|}[b_\mu(\psi_{\lambda,-\partial-\lambda}(a,c))]
-[(\psi_{\lambda,\mu}(a,b))_{\lambda+\mu}c]\\
&&+\psi_{\lambda,-\partial-\lambda}(a,[b_\mu c])-(-1)^{|a||b|}\psi_{\mu,-\partial-\mu}(b,[a_\lambda c])
-\psi_{\lambda+\mu,-\partial-\lambda-\mu}([a_\lambda b],c),
\end{eqnarray*}
which is exactly Eq.\eqref{defor1}. Thus, when $\psi$ is a 2-cocycle satisfying Eq.\eqref{defor2},
 $(\R, [\cdot_{\lambda}\cdot]_{t}, \alpha)$ forms a Lie conformal superalgebra. In this case,
$\psi$ generates a deformation of the Lie conformal superalgebra $\R$.

A deformation is said to be {\it trivial} if there is a linear operator $f\in C^{1}(\R,\R)_{\bar{0}}$ such that for
${T_t}_\lambda={\rm id}+tf_\lambda$, there holds
\begin{eqnarray}\label{1-5}
{T_t}_{-\partial}([a_\lambda b]_t)=[{({T_t}_\lambda(a))}_\lambda {T_t}_{-\partial}(b)], \ \forall \ a,b\in\R.
\end{eqnarray}
\begin{defn}\rm
A linear operator $f\in C^{1}(\R,\R)_{\bar{0}}$ is called a Nijienhuis operator if
\begin{eqnarray}\label{1-4}
[{(f_\lambda(a))}_\lambda (f_\mu(b))]=f_{\lambda+\mu}([a_\lambda b]_N), \ \forall \ a,b\in\R,
\end{eqnarray}
where the bracket $[\cdot_{\lambda}\cdot]_N$ is defined by
\begin{eqnarray}\label{N1}
[a_\lambda b]_N=[(f_\lambda(a))_\lambda b]+[a_\lambda(f_{-\partial}(b))]-f_{-\partial}([a_\lambda b]) \ \forall \ a,b\in\R.\end{eqnarray}
\end{defn}

\begin{re}\rm
In particular, by $(\rm C1)_\lambda$ and setting $\mu=-\partial-\lambda$ in Eq.\eqref{1-4}, we obtain
\begin{eqnarray}\label{4-10}
[{(f_\lambda(a))}_\lambda f_{-\partial}(b)]=f_{-\partial}([a_\lambda b]_N), \ \forall \ a,b\in\R.
\end{eqnarray}
\end{re}

\begin{thm} Let $\R$ be a Lie conformal superalgebra, and $f\in C^{1}(\R,\R)_{\bar{0}}$ a Nijienhuis operator. Then a deformation of $\R$ can be obtained by putting
\begin{eqnarray}\label{N2}
\psi_{\lambda,-\partial-\lambda}(a, b):=({\rm{\bf \hat{d}}}f)_{\lambda,-\partial-\lambda}(a, b)=[a_\lambda b]_{N},  \ \forall \ a,b\in\R.
\end{eqnarray}
Furthermore,  this deformation is trivial.
\end{thm}
\begin{proof}
Since $\psi={\rm{\bf \hat{d}}}f$,  ${\rm{\bf \hat{d}}}\psi=0$ is valid. To see that $\psi$ generates a deformation of $\R$,  we need to check Eq.\eqref{defor2} for $\psi$. By Eqs.\eqref{N1} and \eqref{N2}, we compute and get
\begin{eqnarray*}
&&\psi_{\lambda,-\partial-\lambda}(a, \psi_{\mu,-\partial-\mu}(b,c))=[a_\lambda[b_\mu c]_N]_N\\
&=&[(f_\lambda(a))_\lambda([b_\mu c]_N)]+[a_\lambda(f_{-\partial}([b_\mu c]_N))]-f_{-\partial}([a_\lambda[b_\mu c]_N])\\
&=&[(f_\lambda(a))_\lambda[(f_\mu(b))_\mu c]]+[(f_\lambda(a))_\lambda[b_\mu(f_{-\partial}(c))]]
-[(f_\lambda(a))_\lambda(f_{-\partial}([b_\mu c]))]+[a_\lambda(f_{-\partial}([b_\mu c]_N))]\\
&&-f_{-\partial}([a_\lambda[(f_\mu(b))_\mu c]])-f_{-\partial}([a_\lambda[b_\mu(f_{-\partial}(c))]])+f_{-\partial}([a_\lambda(f_{-\partial}([b_\mu c]))])\\
&=&\underbrace{[(f_\lambda(a))_\lambda[(f_\mu(b))_\mu c]]}_{(1)}+\underbrace{[(f_\lambda(a))_\lambda[b_\mu(f_{-\partial}(c))]]}_{(2)}
-[(f_\lambda(a))_\lambda(f_{-\partial}([b_\mu c]))]+\underbrace{[a_\lambda[(f_\mu(b))_\mu(f_{-\partial}(c))]]}_{(3)}\\
&&\underbrace{-f_{-\partial}([a_\lambda[(f_\mu(b))_\mu c]])}_{(4)}\underbrace{-f_{-\partial}([a_\lambda[b_\mu (f_{-\partial}(c))]])}_{(5)}+f_{-\partial}([a_\lambda(f_{-\partial}([b_\mu c]))]).
\end{eqnarray*}

In the same way, we have
\begin{eqnarray*}
&&(-1)^{|a||b|}\psi_{\mu,-\partial-\mu}(b, \psi_{\lambda,-\partial-\lambda}(a,c))\\
&=&\underbrace{(-1)^{|a||b|}[(f_\mu(b))_\mu[(f_\lambda(a))_\lambda c]]}_{(1)'}+\underbrace{(-1)^{|a||b|}[(f_\mu(b))_\mu[a_\lambda (f_{-\partial}(c))]]}_{(3)'}\\
&&-(-1)^{|a||b|}[(f_\mu(b))_\mu(f_{-\partial}([a_\lambda c]))]+\underbrace{(-1)^{|a||b|}[b_\mu[(f_\lambda(a))_\lambda(f_{-\partial}(c))]]}_{(2)'}\\
&&\underbrace{-(-1)^{|a||b|}f_{-\partial}([b_\mu[(f_\lambda(a))_\lambda c]])}_{(6)'}\underbrace{-(-1)^{|a||b|}f_{-\partial}([b_\mu[a_\lambda (f_{-\partial}(c))]])}_{(5)'}\\
&&+(-1)^{|a||b|}f_{-\partial}([b_\mu(f_{-\partial}([a_\lambda c]))])
\end{eqnarray*}
and
\begin{eqnarray*}
&&\psi_{\lambda+\mu,-\partial-\lambda-\mu}(\psi_{\lambda,-\partial-\lambda}(a,b),c)\\
&=&[(f_{\lambda+\mu}([a_\lambda b]_N))_{\lambda+\mu}c]+[{([a_\lambda b]_N)}_{\lambda+\mu}(f_{-\partial}(c))]-f_{-\partial}([{([a_\lambda b]_N)}_{\lambda+\mu)} c])\\
&=&\underbrace{[[(f_{\lambda}(a))_\lambda(f_{\mu}(b))]_{\lambda+\mu}c]}_{(1)''}\\
&&+\underbrace{[[(f_\lambda(a))_\lambda b]_{\lambda+\mu}(f_{-\partial}(c))]}_{(2)''}+\underbrace{[[a_\lambda (f_{-\partial}(b))]_{\lambda+\mu}(f_{-\partial}(c))]}_{(3)''}-[(f_{-\partial}([a_\lambda b]))_{\lambda+\mu}(f_{-\partial}(c))]\\
&&\underbrace{-f_{-\partial}([(f_\lambda(a))_\lambda b]_{\lambda+\mu}c])}_{(6)''}\underbrace{-f_{-\partial}([[a_\lambda (f_{-\partial}(b))]_{\lambda+\mu}}_{(4)''} c])+f_{-\partial}([(f_{-\partial}([a_\lambda b]))_{\lambda+\mu} c]).
\end{eqnarray*}

Since $f$ is a Nijienhuis operator and by Eq.\eqref{4-10}, we get
\begin{eqnarray*}
&&\ \ \ \ \,-[(f_\lambda(a))_\lambda(f_{-\partial}([b_\mu c]))]+f_{-\partial}([a_\lambda(f_{-\partial}([b_\mu c]))])\\
&&\ \ \ \ \ \ \ \ \ \ \ \ \ \ \ \ \ \ \ \,=\underbrace{-f_{-\partial}([(f_\lambda(a))_\lambda[b_\mu c]])}_{(6)}+\underbrace{f_{-\partial}^2([a_\lambda [b_\mu c]])}_{(7)},\\
&&\ \ \ \ \,-(-1)^{|a||b|}[(f_\mu(b))_\mu(f_{-\partial}([a_\lambda c]))]+(-1)^{|a||b|}f_{-\partial}([b_\mu(f_{-\partial}([a_\lambda c]))])\\
&&\ \ \ \ \ \ \ \ \ \ \ \ \ \ \ \ \ \ \ \,=\underbrace{-(-1)^{|a||b|}f_{-\partial}([(f_\mu(b))_\mu [a_\lambda c]])}_{(4)^{'}}
+\underbrace{(-1)^{|a||b|}f_{-\partial}^2([b_\mu [a_\lambda c]])}_{(7)^{'}}.
\end{eqnarray*}
By $(\rm C1)_\lambda$ and Eq.\eqref{4-10},
\begin{eqnarray*}
&&-[(f_{-\partial}([a_\lambda b]))_{\lambda+\mu}(f_{-\partial}(c))]+f_{-\partial}([(f_{-\partial}([a_\lambda b]))_{\lambda+\mu} c])\\
&=&-[(f_{\lambda+\mu}([a_\lambda b]))_{\lambda+\mu}(f_{-\partial}(c))]+f_{-\partial}([(f_{\lambda+\mu}([a_\lambda b]))_{\lambda+\mu} c])\\
&=&\underbrace{-f_{-\partial}([[a_\lambda b]_{\lambda+\mu}(f_{-\partial}(c))])}_{(5)^{''}}+\underbrace{f_{-\partial}^2([[a_\lambda b]_{\lambda+\mu}c])}_{(7)^{''}}.
\end{eqnarray*}

Note that according to the Jacobi identity and $(\rm C1)_\lambda$ for $a ,b , c\in\R$,
$$[a_\lambda[(f_\mu(b))_\mu(f_{-\partial}(c))]]=[[a_\lambda f_\mu(b)]_{\lambda+\mu}(f_{-\partial}(c))]+(-1)^{|a||b|}[(f_\mu(b))_\mu[a_\lambda (f_{-\partial}(c))]]$$
is equivalent to
\begin{eqnarray*}[a_\lambda[(f_\mu(b))_\mu(f_{-\partial}(c))]]=[[a_\lambda f_{-\partial}(b)]_{\lambda+\mu}(f_{-\partial}(c))]+(-1)^{|a||b|}[(f_\mu(b))_\mu[a_\lambda(f_{-\partial}(c))]].\end{eqnarray*}
Thus $(i)+(i)^{'}+(i)^{''}=0$, for $i=1,\cdots,7$. This proves that
$\psi$ generates a deformation of the Lie conformal superalgebra $\R$.

Let ${T_{t}}_\lambda=\id+tf_\lambda$. By Eqs.\eqref{1-3} and \eqref{N2},
\begin{eqnarray}\label{1-1}
{T_{t}}_{-\partial}([a_\lambda b]_{t})
&=&(\id+tf_{-\partial})([a_\lambda b]+t\psi_{\lambda,-\partial-\lambda}(a,b))\nonumber\\
&=&(\id+tf_{-\partial})([a_\lambda b]+t[a_\lambda b]_{N})\nonumber\\
&=&[a_\lambda b]+t([a_\lambda b]_{N}+f_{-\partial}([a_\lambda b]))+t^{2}f_{-\partial}([a_\lambda b]_{N}).
\end{eqnarray}
Furthermore,
\begin{eqnarray}\label{1-2}
[({T_{t}}_\lambda(a))_\lambda {T_{t}}_{-\partial}(b)]
&=&[(a+tf_\lambda(a))_\lambda(b+tf_{-\partial}(b))]\nonumber\\
&=&[a_\lambda b]+t([(f_\lambda(a))_\lambda b]+[a_\lambda(f_{-\partial}(b))])+t^{2}[(f_\lambda(a))_\lambda f_{-\partial}(b)].
\end{eqnarray}
Combining Eqs.\eqref{1-1} with \eqref{1-2} gives
${T_{t}}_{-\partial}([a_\lambda b]_{t})=[({T_{t}}_\lambda(a))_\lambda {T_{t}}_{-\partial}(b)].$
Therefore the deformation is trivial.
\end{proof}

\section{Generalized derivations of Lie conformal superalgebras}

\par

Let $\R$ be a Lie conformal superalgebra. Define $\texttt{\Der}(\R)_\theta$ is the set of all derivations of degree $\theta$, then it is obvious that $\texttt{\Der}(\R)=\texttt{\Der}(\R)_{\bar{0}}\oplus\texttt{\Der}(\R)_{\bar{1}}$ is a subalgebra of $\Cend(\R)$.

\begin{defn}\rm An element $f$ in $\Cend(\R)_\theta$ is called
\begin{itemize}
\item a {\it generalized derivation} of degree $\theta$ of $\R$, if there exist $f^{'},f^{''}\in\Cend(\R)_\theta$ such that
\begin{eqnarray}\label{5-1}
[(f_\lambda(a))_{\lambda+\mu}b]+(-1)^{\theta|a|}[a_\mu(f^{'}_\lambda(b))]=f^{''}_\lambda([a_\mu b]), \ \forall \ a,b\in\R.
\end{eqnarray}
\item an {\it quasiderivation }of degree $\theta$ of $\R$, if there is $f^{'}\in\Cend(\R)_\theta$ such that
\begin{eqnarray}\label{5-2}
[(f_\lambda(a))_{\lambda+\mu}b]+(-1)^{\theta|a|}[a_\mu(f_\lambda(b))]=f^{'}_\lambda([a_\mu b]), \ \forall \ a,b\in\R.
\end{eqnarray}
\item an {\it centroid} of degree $\theta$ of $\R$, if it satisfies
\begin{eqnarray}\label{5-3}
 [(f_\lambda(a))_{\lambda+\mu}b]=(-1)^{\theta|a|}[a_\mu(f_\lambda(b))]=f_\lambda([a_\mu b]), \ \forall \ a,b\in\R.
\end{eqnarray}
\item an {\it quasicentroid} of degree $\theta$ of $\R$, if it satisfies
\begin{eqnarray}\label{5-4}
[(f_\lambda(a))_{\lambda+\mu}b]=(-1)^{\theta|a|}[a_\mu(f_\lambda(b))], \ \forall \ a,b\in\R.
\end{eqnarray}
\item an {\it central derivation} of degree $\theta$ of $\R$, if it satisfies
\begin{eqnarray}\label{5-5}
[(f_\lambda(a))_{\lambda+\mu}b]=f_\lambda([a_\mu b])=0, \ \forall \ a,b\in\R.
\end{eqnarray}
\end{itemize}
\end{defn}

Denote by $\GDer(\R)_\theta$, $\QDer(\R)_\theta$, $\C(\R)_\theta$, $\QC(\R)_\theta$ and $\ZDer(\R)_\theta$ the sets of all generalized derivations,
 quasiderivations, centroids, quasicentroids and central derivations of degree $\theta$ of $\R$.
It is easy to see that
\begin{eqnarray}
\ZDer(\R)\subseteq \Der(\R)\subseteq \QDer(\R)\subseteq \GDer(\R)\subseteq \Cend(\R),\,
\C(\R)\subseteq \QC(\R)\subseteq \GDer(\R).\nonumber\\ \label{tower}
\end{eqnarray}

\begin{prop}Let $\R$ be a Lie conformal superalgebra. Then
\begin{enumerate}
\item[$(1)$] $\GDer(\R)$, $\QDer(\R)$ and $\C(\R)$ are subalgebras of $\Cend(\R)$.
\item[$(2)$] $\ZDer(\R)$ is an ideal of $\Der(\R)$.
\end{enumerate}
\end{prop}
\begin{proof} $(1)$  We only prove that $\GDer(\R)$ is a subalgebra of $\Cend(\R)$. The proof for the other two cases is exactly analogous.

For $f\in\GDer(\R)_\sigma, g\in\GDer(\R)_\vartheta$, $a,b\in \R$, there exist $f^{'},f^{''}\in\GDer(\R)_\sigma$  (resp. $g^{'},g^{''}\in\GDer(\R)_\vartheta$ ) such that Eq.\eqref{5-1} holds for $f$ (resp. $g$).
We only need to show
\begin{eqnarray}\label{5-6}
[f^{''}_\lambda g^{''}]_\theta([a_\mu b])=[([f_\lambda g]_\theta(a))_{\mu+\theta}b]+(-1)^{(\sigma+\vartheta)|a|}[a_\mu([f^{'}_\lambda g^{'}]_\theta(b))].
\end{eqnarray}

By Eq.\eqref{def2-5}, we have
\begin{eqnarray}
[([f_\lambda g]_\theta(a))_{\mu+\theta}b]
=[(f_\lambda(g_{\theta-\lambda}(a)))_{\mu+\theta}b]-(-1)^{\sigma\vartheta}[(g_{\theta-\lambda}(f_\lambda(a)))_{\mu+\theta}b]. \label{5-7}
\end{eqnarray}

By Eq.\eqref{5-1}, we obtain
\begin{eqnarray}
&&[(f_\lambda(g_{\theta-\lambda}(a)))_{\mu+\theta}b]\nonumber\\
&=&f^{''}_\lambda([(g_{\theta-\lambda}(a))_{\mu+\theta-\lambda}b])-(-1)^{\sigma(\vartheta+|a|)}[(g_{\theta-\lambda}(a))_{\mu+\theta-\lambda}(f^{'}_\lambda(b))]\nonumber\\
&=&f^{''}_\lambda(g^{''}_{\theta-\lambda}([a_\mu b]))-(-1)^{\vartheta|a|}f^{''}_\lambda([a_\mu (g^{'}_{\theta-\lambda}(b))])
\nonumber\\&&-(-1)^{\sigma(\vartheta+|a|)}g^{''}_{\theta-\lambda}([a_\mu(f^{'}_\lambda(b))])+(-1)^{\sigma(\vartheta+|a|)
+\vartheta|a|}[a_\mu(g^{'}_{\theta-\lambda}(f^{'}_\lambda(b)))],\label{5-8}\\
&&[(g_{\theta-\lambda}(f_\lambda(a)))_{\mu+\theta}b]\nonumber\\
&=&g^{''}_{\theta-\lambda}([(f_\lambda(a))_{\lambda+\mu}b])-(-1)^{\vartheta(\sigma+|a|)}[f_\lambda(a)_{\lambda+\mu}(g^{'}_{\theta-\lambda}(b))]\nonumber\\
&=&g^{''}_{\theta-\lambda}(f^{''}_\lambda([a_\mu b]))-(-1)^{\sigma|a|}g^{''}_{\theta-\lambda}([a_\mu(f^{'}_\lambda(b))])\nonumber\\
&&-(-1)^{\vartheta(\sigma+|a|)}f^{''}_\lambda([a_\mu(g^{'}_{\theta-\lambda}(b))])+(-1)^{\vartheta(\sigma+|a|)
+\sigma|a|}[a_\mu(f^{'}_\lambda(g^{'}_{\theta-\lambda}(b))).\label{5-9}
\end{eqnarray}

Substituting Eqs.\eqref{5-8} and \eqref{5-9} into Eq.\eqref{5-7} gives Eq.\eqref{5-6}. Hence $[f_\lambda g]\in\GDer(\R)[\lambda]$,
and $\GDer(\R)$ is a subalgebra of $\Cend(\R)$.

(2) For $f\in\ZDer(\R), g\in\Der(\R)$, and $a, b \in\R$, by Eq.\eqref{5-5}, we have
\begin{eqnarray*}
[f_\lambda g]_\theta([a_\mu b])
&=&f_\lambda(g_{\theta-\lambda}([a_\mu b]))-(-1)^{|f||g|}g_{\theta-\lambda}(f_\lambda([a_\mu b]))=f_\lambda(g_{\theta-\lambda}([a_\mu b]))\\
&=&f_\lambda([(g_{\theta-\lambda}(a))_{\mu+\theta-\lambda}b]+(-1)^{|a||g|}[a_\mu (g_{\theta-\lambda}(b))])=0,\\
{[[f_\lambda g]_\theta(a)_{\mu+\theta}b]}
&=&[(f_\lambda(g_{\theta-\lambda}(a))-(-1)^{|f||g|}g_{\theta-\lambda}(f_\lambda(a)))_{\mu+\theta}b]\\
&=&[-(-1)^{|f||g|}(g_{\theta-\lambda}(f_\lambda a)))_{\mu+\theta}b]\\
&=&-(-1)^{|f||g|}g_{\theta-\lambda}([f_\lambda(a)_{\lambda+\mu}b])+(-1)^{|g||a|}[f_\lambda(a)_{\lambda+\mu}g_{\theta-\lambda}(b)]\\
&=&0.
\end{eqnarray*}
This shows that
$[f_\lambda g] \in\ZDer(\R)[\lambda]$. Thus $\ZDer(\R)$ is an ideal of $\Der(\R)$.
\end{proof}

\begin{lem} \label{lemm5-1}
Let $\R$ be a Lie conformal superalgebra. Then
\begin{enumerate}
\item[$(1)$] $[\Der(\R)_\lambda\C(\R)]\subseteq\C(\R)[\lambda]$,
\item[$(2)$] $[\QDer(\R)_\lambda\QC(\R)]\subseteq\QC(\R)[\lambda]$,
\item[$(3)$] $[\QC(\R)_\lambda\QC(\R)]\subseteq\QDer(\R)[\lambda]$.
\end{enumerate}
\end{lem}
\begin{proof} It is straightforward. \end{proof}

\begin{thm}
Let $\R$ be a Lie conformal superalgebra. Then $$\GDer(\R)=\QDer(\R)+\QC(\R).$$
\end{thm}
\begin{proof} For $f\in\GDer(\R)_\theta$, there exist $f^{'}, f^{''}\in\Cend(\R)_\theta$ such that
\begin{eqnarray}\label{5-10}
[(f_\lambda(a))_{\lambda+\mu}b]+(-1)^{\theta|a|}[a_\mu(f^{'}_\lambda(b))]=f^{''}_\lambda([a_\mu b]), \forall \, a,b\in\R.
\end{eqnarray}
By $(\rm C2)_\lambda$ and Eq.(\ref{5-10}), we get
\begin{eqnarray}
(-1)^{\theta|b|}[b_{-\partial-\lambda-\mu}(f_\lambda(a))]+[(f^{'}_\lambda(b))_{-\partial-\mu}a]=f^{''}_\lambda([b_{-\partial-\mu}a]).\label{represen6}
\end{eqnarray}
By $(\rm C1)_\lambda$ and setting $\mu=-\partial-\lambda-\mu'$ in Eq.(\ref{represen6}), we obtain
\begin{eqnarray}
(-1)^{\theta|b|}[b_{\mu'}(f_\lambda(a))]+[(f^{'}_\lambda(b))_{\lambda+\mu'}a]=f^{''}_\lambda([b_{\mu'}a]).\label{5-11}
\end{eqnarray}
Then, changing the place of $a,b$ and replacing $\mu'$ by $\mu$ in Eq.(\ref{5-11}) give
\begin{eqnarray}
(-1)^{\theta|a|}[a_\mu(f_\lambda(b))]+[(f^{'}_\lambda(a)))_{\lambda+\mu}b]=f^{''}_\lambda([a_\mu b]).\label{5-12}
\end{eqnarray}

Combining Eqs.(\ref{5-10}) and (\ref{5-12}) gives
\begin{eqnarray*}
&&[(\frac{f_\lambda+f^{'}_\lambda}{2}(a))_{\lambda+\mu}b]+(-1)^{\theta|a|}[a_{\mu}(\frac{f_\lambda+f^{'}_\lambda}{2}(b))]=f^{''}_\lambda([a_\mu b]),\\
&&[(\frac{f_\lambda-f^{'}_\lambda}{2}(a))_{\lambda+\mu}b]-(-1)^{\theta|a|}[a_{\mu}(\frac{f_\lambda-f^{'}_\lambda}{2}(b))]=0.
\end{eqnarray*}
It follows that $\frac{f_\lambda+f^{'}_\lambda}{2}\in\QDer(\R)$ and $\frac{f_\lambda-f^{'}_\lambda}{2}\in\QC(\R)$. Hence
$$f_\lambda=\frac{f_\lambda+f^{'}_\lambda}{2}+\frac{f_\lambda-f^{'}_\lambda}{2}\in\QDer(\R)+\QC(\R),$$
proving that $\GDer(\R)\subseteq\QDer(\R)+\QC(\R).$  The reverse inclusion relation follows from Eq.\eqref{tower} and Lemma \ref{lemm5-1}.
\end{proof}

\begin{thm}
Let $\R$ be a Lie conformal superalgebra and $\Z(\R)$ the center of $\R$. Then
$[\C(\R)_\lambda\QC(\R)]\subseteq\Chom(\R,\Z(\R))[\lambda]$. Moreover, if $\Z(\R)=0$, then $[\C(\R)_\lambda\QC(\R)]=0$.
\end{thm}
\begin{proof}
For $f\in\C(\R),g\in\QC(\R)$, and $a,b\in\R$,
by Eqs.\eqref{5-3}and \eqref{5-4}, we have
\begin{eqnarray*}
[([f_\lambda g]_\theta(a))_{\mu+\theta}b]
&=&[(f_\lambda(g_{\theta-\lambda}(a)))_{\mu+\theta}b]
-(-1)^{|f||g|}[(g_{\theta-\lambda}(f_\lambda(a)))_{\mu+\theta}b]\\
&=&f_\lambda([(g_{\theta-\lambda}(a))_{\mu+\theta-\lambda}b])
-(-1)^{|f||g|+|g|(|f|+|a|)}[(f_\lambda(a))_{\lambda+\mu}(g_{\theta-\lambda}(b))]\\
&=&f_\lambda([(g_{\theta-\lambda}(a))_{\mu+\theta-\lambda}b])
-(-1)^{|g||a|}f_\lambda([a_\mu(g_{\theta-\lambda}(b))])\\
&=&f_\lambda([(g_{\theta-\lambda}(a))_{\mu+\theta-\lambda}b]
-(-1)^{|g||a|}[a_\mu(g_{\theta-\lambda}(b))])\\
&=&0.
\end{eqnarray*}
Hence $[f_\lambda g](a)\in \Z(\R)[\lambda]$, and then $[f_\lambda g]\in\Chom(\R,\Z(\R))[\lambda]$. If $\Z(\R)=0$, then $[f_\lambda g](a)=0$, $\forall$ $a\in\hg(\R).$ Thus $[\C(\R)_\lambda\QC(\R)]=0$.
\end{proof}

\begin{lem}\label{lemm5-6}
Let $\R$ be a Lie conformal superalgebra. If $f\in\QC(\R)_\theta \cap\QDer(\R)_\theta$ and $f,f'$ satisfy {\rm Eq.\eqref{5-2}}, then we have
\begin{eqnarray*}
[a_\lambda(f_\gamma([b_\mu c]))]=[a_\lambda[(f_\gamma(b))_{\mu+\gamma}c]]=(-1)^{\theta|b|}[a_\lambda[b_\mu(f_\gamma(c))]].
\end{eqnarray*}
\end{lem}
\begin{proof}
Let $f'=2\varphi$. Since $f\in\QC(\R)_\theta \cap\QDer(\R)_\theta$ and $f,f'$ satisfy Eq.\eqref{5-2}, we get
\begin{eqnarray}
&&[(f_\gamma(a))_{\lambda+\gamma}b]=(-1)^{\theta|a|}[a_\lambda(f_\gamma(b))],\nonumber\\
&&\varphi_\gamma([a_\lambda b])=\frac{1}{2}f'_\gamma([a_\lambda b])=\frac{1}{2}([(f_\gamma(a))_{\lambda+\gamma}b]+(-1)^{\theta|a|}[a_\lambda(f_\gamma(b))])\nonumber\\
&=&[(f_\gamma(a))_{\lambda+\gamma}b]=(-1)^{\theta|a|}[a_\lambda(f_\gamma(b))].\label{5-18}
\end{eqnarray}
By Eq.\eqref{5-18}, we have
\begin{eqnarray}
[(\varphi_\gamma([a_\lambda b]))_{\lambda+\mu+\gamma}c]
&=&[[(f_\gamma(a))_{\lambda+\gamma}b]_{\lambda+\mu+\gamma}c]\nonumber\\
&=&[(f_\gamma(a))_{\lambda+\gamma}[b_\mu c]]-(-1)^{(\theta+|a|)||b|}[b_\mu[(f_\gamma(a))_{\lambda+\gamma}c]]\nonumber\\
&=&[(f_\gamma(a))_{\lambda+\gamma}[b_\mu c]]+(-1)^{|b||c|}[[(f_\gamma(a))_{\lambda+\gamma}c]_{-\partial-\mu}b]\nonumber\\
&=&\varphi_\gamma([a_\lambda[b_\mu c]])+(-1)^{|b||c|}[(\varphi_\gamma([a_\lambda c])_{-\partial-\mu}b]\nonumber\\
&=&\varphi_\gamma([a_\lambda[b_\mu c]])-(-1)^{(|a|+|b|)|c|}[(\varphi_\gamma([c_{-\partial-\lambda}a])_{-\partial-\mu}b],\label{5-19}
\end{eqnarray}
it is equivalent to
\begin{eqnarray}
(-1)^{|a||c|}[(\varphi_\gamma([a_\lambda b]))_{\lambda+\mu+\gamma}c]+(-1)^{|b||c|}[(\varphi_\gamma([c_{-\partial-\lambda}a])_{-\partial-\mu}b]
=(-1)^{|a||c|}\varphi_\gamma([a_\lambda[b_\mu c]]).\label{5-20}
\end{eqnarray}

On the one hand, by $(\rm C1)_\lambda$ and setting $\mu=-\partial-\lambda-\gamma-\mu'$ in Eq.\eqref{5-20}, we obtain
\begin{eqnarray*}
(-1)^{|a||c|}[(\varphi_\gamma([a_\lambda b]))_{-\partial-\mu'}c]+(-1)^{|b||c|}[\varphi_\gamma([c_{-\partial-\lambda}a])_{\lambda+\gamma+\mu'}b]
=(-1)^{|a||c|}\varphi_\gamma([a_\lambda[b_{-\partial-\mu'}c]]),
\end{eqnarray*}
it is equivalent to
\begin{eqnarray}
(-1)^{|a||c|}[\varphi_\gamma([a_\lambda b])_{-\partial-\mu'}c]+(-1)^{|b||c|}[\varphi_\gamma([c_{\mu'}a])_{\lambda+\gamma+\mu'}b]
=(-1)^{|a||c|}\varphi_\gamma([a_\lambda[b_{-\partial-\mu'}c]]).\label{5-21}
\end{eqnarray}
By $(\rm C1)_\lambda$ and setting $\lambda=-\partial-\lambda'-\gamma-\mu'$ in Eq.\eqref{5-21}, we get
\begin{eqnarray}
(-1)^{|a||c|}[\varphi_\gamma([a_{-\partial-\lambda'}b])_{-\partial-\mu'}c]+(-1)^{|b||c|}[\varphi_\gamma([c_{\mu'}a])_{-\partial-\lambda'}b]\nonumber\\
=(-1)^{|a||c|}\varphi_\gamma([a_{-\partial-\lambda'-\mu'}[b_{-\partial-\mu'}c]]).\label{5-22}
\end{eqnarray}

On the other hand, replacing $\mu, \lambda$ by $\lambda, \mu$ in Eq.\eqref{5-20} give
\begin{eqnarray}
(-1)^{|a||c|}[(\varphi_\gamma([a_\mu b]))_{\lambda+\mu+\gamma}c]+(-1)^{|b||c|}[(\varphi_\gamma([c_{-\partial-\mu}a])_{-\partial-\lambda}b]
=(-1)^{|a||c|}\varphi_\gamma([a_\mu[b_\lambda c]]).\label{5-23}
\end{eqnarray}
By $(\rm C1)_\lambda$ and setting $\lambda=-\partial-\lambda'-\gamma-\mu$ in Eq.\eqref{5-23}, we get
\begin{eqnarray*}
(-1)^{|a||c|}[(\varphi_\gamma([a_\mu b]))_{-\partial-\lambda'}c]+(-1)^{|b||c|}[(\varphi_\gamma([c_{-\partial-\mu}a])_{\lambda'+\gamma+\mu}b]
=(-1)^{|a||c|}\varphi_\gamma([a_\mu[b_{-\partial-\lambda'}c]]),
\end{eqnarray*}
it is equivalent to
\begin{eqnarray}
(-1)^{|a||c|}[(\varphi_\gamma([a_\mu b]))_{-\partial-\lambda'}c]+(-1)^{|b||c|}[(\varphi_\gamma([c_{\lambda'}a])_{\lambda'+\gamma+\mu}b]
=(-1)^{|a||c|}\varphi_\gamma([a_\mu[b_{-\partial-\lambda'}c]]).\label{5-24}
\end{eqnarray}

By Eqs.\eqref{5-20}, \eqref{5-22} and \eqref{5-24}, we have
\begin{eqnarray*}
&&2((-1)^{|a||c|}[(\varphi_\gamma([a_\lambda b]))_{\lambda+\mu+\gamma}c]+(-1)^{|b||c|}[(\varphi_\gamma([c_{-\partial-\lambda}a])_{-\partial-\mu}b]\\
&&+(-1)^{|a||b|}[(\varphi_\gamma([b_\mu c]))_{-\partial-\lambda}a])\\
&=&(-1)^{|a||c|}[(\varphi_\gamma([a_\lambda b]))_{\lambda+\mu+\gamma}c]+(-1)^{|b||c|}[(\varphi_\gamma([c_{-\partial-\lambda}a])_{-\partial-\mu}b]\\
&&+(-1)^{|b||c|}[(\varphi_\gamma([c_{-\partial-\lambda}a])_{-\partial-\mu}b]+(-1)^{|a||b|}[(\varphi_\gamma([b_\mu c]))_{-\partial-\lambda}a]\nonumber\\
&&+(-1)^{|a||b|}[(\varphi_\gamma([b_\mu c]))_{-\partial-\lambda}a]+(-1)^{|a||c|}[(\varphi_\gamma([a_\lambda b]))_{\lambda+\mu+\gamma}c]\\
&=&(-1)^{|a||c|}\varphi_\gamma([a_\lambda[b_\mu c]])+(-1)^{|b||c|}\varphi_\gamma([c_{-\partial-\lambda-\mu}[a_{-\partial-\mu}b]])
+(-1)^{|a||b|}\varphi_\gamma([b_\mu[c_{-\partial-\lambda}a]])\\
&=&(-1)^{|a||c|}\varphi_\gamma([a_\lambda[b_\mu c]])+(-1)^{|b||c|}\varphi_\gamma([c_{-\partial-\lambda-\mu}[a_{\lambda}b]])
+(-1)^{|a||b|}\varphi_\gamma([b_\mu[c_{-\partial-\lambda}a]])\\
&=&0,
\end{eqnarray*}so
\begin{eqnarray}
(-1)^{|a||c|}[(\varphi_\gamma([a_\lambda b]))_{\lambda+\mu+\gamma}c]+(-1)^{|b||c|}[(\varphi_\gamma([c_{-\partial-\lambda}a])_{-\partial-\mu}b]\nonumber\\
+(-1)^{|a||b|}[(\varphi_\gamma([b_\mu c]))_{-\partial-\lambda}a]=0. \label{5-25}
\end{eqnarray}

Combining Eqs.\eqref{5-20} with \eqref{5-25}, we get
\begin{eqnarray*}
(-1)^{|a||c|}\varphi_\gamma([a_\lambda[b_\mu c]])=-(-1)^{|a||b|}[(\varphi_\gamma([b_\mu c]))_{-\partial-\lambda}a]=(-1)^{|a||c|+\theta|a|}[a_\lambda(\varphi_\gamma([b_\mu c]))],
\end{eqnarray*}
that is
\begin{eqnarray}
\varphi_\gamma([a_\lambda[b_\mu c]])=(-1)^{\theta|a|}[a_\lambda(\varphi_\gamma([b_\mu c]))].\label{5-26}
\end{eqnarray}
By Eqs.\eqref{5-26} and \eqref{5-18}, we obtain
\begin{eqnarray*}
[a_\lambda(f_\gamma([b_\mu c]))]=(-1)^{\theta|a|}\varphi_\gamma([a_\lambda[b_\mu c]]=[a_\lambda(\varphi_\gamma([b_\mu c]))]\nonumber\\
=[a_\lambda[(f_\gamma(b))_{\mu+\gamma}c]]=(-1)^{\theta|b|}[a_\lambda[b_\mu(f_\gamma(c))]].
\end{eqnarray*}
This shows the lemma.
\end{proof}
\begin{lem}\label{lemm5-7}
Let $\R$ be a Lie conformal superalgebra. If $f\in\QC(\R)_\theta\cap\QDer(\R)_\theta$ and $\Z(\R)=0$, then $f\in\C(\R)_\theta$.
\end{lem}
\begin{proof}
It is straightforward by Lemma \ref{lemm5-6}.
\end{proof}

\begin{thm}
Let $\R$ be a Lie conformal superalgebra. If $\Z(\R)=0$, then $\QC(\R)\cap\QDer(\R)=\C(\R)$.
\end{thm}
\begin{proof}
$\C(\R)\subseteq\QC(\R)\cap\QDer(\R)$ is obvious. By Lemma \ref{lemm5-7}, if $\Z(\R)=0$, we have
$\QC(\R)_\theta\cap\QDer(\R)_\theta\subseteq\C(\R)_\theta$. So $\QC(\R)\cap\QDer(\R)=\C(\R)$.
\end{proof}

\begin{prop}
Let $\R$ be a Lie conformal superalgebra. If $\Z(\R)=0$, then $\QC(\R)$ is a Lie conformal superalgebra if and  only if $[\QC(\R)_\lambda\QC(\R)]=0$.
\end{prop}
\begin{proof}
$(\Rightarrow)$  Suppose that $\QC(\R)$ is a Lie conformal superalgebra.
For $f\in\QC(\R)_\theta,g\in\QC(\R)_\vartheta$, $[f_\lambda g]\in\QC(\R)_{\theta+\vartheta}[\lambda]$. For $a,b\in \R$, by Eq.\eqref{5-4}, we have
\begin{eqnarray}\label{5-15}[([f_\lambda g]_\theta(a))_{\mu+\theta}b]
=(-1)^{(\theta+\vartheta)|a|}[a_\mu([f_\lambda g]_\theta(b))].\end{eqnarray}
By Eqs.\eqref{def2-5} and \eqref{5-4}, we obtain
\begin{eqnarray}\label{5-16}
&&[([f_\lambda g]_\theta(a))_{\mu+\theta}b]\nonumber\\
&=&[(f_\lambda(g_{\theta-\lambda}(a)))_{\mu+\theta}b]-(-1)^{\theta\vartheta}[(g_{\theta-\lambda}(f_\lambda(a)))_{\mu+\theta}b]\nonumber\\
&=&(-1)^{\theta(\vartheta+|a|)}[(g_{\theta-\lambda}(a))_{\mu+\theta-\lambda}(f_\lambda(b))]
-(-1)^{\theta\vartheta+\vartheta(\theta+|a|)}[(f_\lambda(a))_{\lambda+\mu}(g_{\theta-\lambda}(b))]\nonumber\\
&=&(-1)^{\theta(\vartheta+|a|)+\vartheta|a|}[a_\mu(g_{\theta-\lambda}(f_\lambda(b)))]-(-1)^{\theta\vartheta+\vartheta(\theta+|a|)+\theta|a|}[a_\lambda(f_\lambda(g_{\theta-\lambda}(b)))]\nonumber\\
&=&-(-1)^{(\vartheta+\theta)|a|}[a_\mu([f_\lambda g]_\theta(b))].
\end{eqnarray}
Combining Eqs.\eqref{5-15} and \eqref{5-16} gives
$$[([f_\lambda g]_\theta(a))_{\mu+\theta}b]=0,$$
and thus $[f_\lambda g]_\theta(a)\in\Z(\R)[\lambda]=0$($\Z(\R)=0$).
Therefore, $[f_\lambda g]=0$.

$(\Leftarrow)$ It is clear.
\end{proof}


\begin{thebibliography}{99}

\bibitem{BKV}  B. Bakalov, V. Kac, A. Voronov, Cohomology of conformal algebras,
\emph{Comm. Math. Phys.} \textbf{200} (1999), no. 3, 561-598.


\bibitem{BKLR}  C. Boyallian,  V. Kac,  J. Liberati, A. Rudakov, Representations of simple finite Lie conformal superalgebras of type W  and S. \emph{J. Math. Phys.} \textbf{47} (2006), no. 4, 043513, 25 pp.

\bibitem{Z}  Z. Chang, Automorphisms and twisted forms of differential lie conformal superalgebras. \emph{Thesis (Ph.D.)-University of Alberta} (Canada). 2013. 117 pp.

\bibitem{DK}  A. D'Andrea, V. Kac, Structure theory of finite conformal algebras, \emph{Selecta Math. (N.S.)} \textbf{4} (1998), no. 3, 377-418.



\bibitem{FKR}  D. Fattori, V. Kac, A. Retakh, Structure theory of finite Lie conformal superalgebras. \emph{Lie theory and its applications in physics V}, 27-63, World Sci. Publ., River Edge, NJ, 2004.

\bibitem{FHS} G. Fan, Y. Hong, Y. Su, Generalized conformal derivations of Lie conformal algebras, arXiv:1602.01159 (2016).

\bibitem{GT} J. Guo, Y. Tan, Conformal derivations of semidirect products of Lie conformal algebras and their conformal modules. \emph{Proc. Amer. Math. Soc.} \textbf{142} (2014), no. 5, 1471¨C1483.

\bibitem{L}  G. Leger, E. Luks, Generalized derivations of Lie algebras, \emph{J. Algebra}, \textbf{228} (2000) 165-203.

\bibitem{LCM} Y. Liu, L. Chen, Y. Ma, Hom-Nijienhuis operators and $T$*-extensions of hom-Lie superalgebras, \emph{Linear Algebra Appl.} \textbf{439} (2013), no. 7, 2131-2144.


















\end{thebibliography}
\end{document}